\documentclass[11pt]{article}
\usepackage{jhtitle,amsbsy,amsmath,amsthm,amssymb,times,graphicx,mathabx}
\usepackage{probmacros}
\usepackage[sort]{natbib}

\newcommand{\pcite}[1]{\citeauthor{#1}'s \citeyearpar{#1}}

\newcommand{\Ex}{\mathrm{E}}
\newcommand{\Var}{\mathrm{Var}}
\newcommand{\tr}{\mathrm{tr}}
\newcommand{\rank}{\mathrm{rank}}
\newcommand{\ind}{\stackrel{\mathrm{ind}}{\sim}}
\newcommand{\iid}{\stackrel{\mathrm{iid}}{\sim}}

\def\baro{\vskip  .2truecm\hfill \hrule height.5pt \vskip  .2truecm}
\def\barba{\vskip -.1truecm\hfill \hrule height.5pt \vskip .4truecm}

\newtheorem{proposition}{Proposition}

\newtheorem{lemma}{Lemma}

\newtheorem{remark}{Remark}

\theoremstyle{remark}

\newcommand{\X}{{\mathsf{X}}}

\newcommand{\tsl}{T_{\lambda,\tau}}
\newcommand{\tsli}{T_{\lambda,\tau}^{-1}}

\newcommand{\msl}{M_{\lambda, \tau}}
\newcommand{\qsl}{Q_{\lambda,\tau}}
\newcommand{\qsli}{Q_{\lambda,\tau}^{-1}}

\begin{document}
\title{Fast Monte Carlo Markov chains \\ for Bayesian shrinkage models
  with random effects} \author{Tavis Abrahamsen \\ Department of
  Statistical Science \\ Duke University\\ tavis.abrahamsen@duke.edu \and James P. Hobert \\
  Department of Statistics \\ University of Florida\\ jhobert@stat.ufl.edu}


\maketitle

\begin{abstract}
When performing Bayesian data analysis using a general linear mixed
model, the resulting posterior density is almost always analytically
intractable.  However, if proper conditionally conjugate priors are used, there is a
simple two-block Gibbs sampler that is geometrically ergodic in nearly
all practical settings, including situations where $p > n$
\citep{abra:hobe:2017}.  Unfortunately, the (conditionally conjugate)
multivariate normal prior on $\beta$ does not perform well in the
high-dimensional setting where $p \gg n$.  In this paper, we consider an
alternative model in which the multivariate normal prior is replaced by
the normal-gamma shrinkage prior developed by
\cite{griffin2010inference}.  This change leads to a much more complex
posterior density, and we develop a simple MCMC algorithm for exploring
it.  This algorithm, which has both deterministic and random scan
components, is easier to analyze than the more obvious three-step Gibbs
sampler.  Indeed, we prove that the new algorithm is geometrically
ergodic in most practical settings.
\end{abstract}

\section{Introduction}

The general linear mixed model (or variance components model) is one
of the most frequently applied statistical models.  It takes the form
\begin{equation*}
  Y = X \beta + \sum_{i=1}^m Z_i u_i + e \;,
\end{equation*}
where $Y$ is an observable $n \times 1$ data vector, $X$ and
$\{Z_i\}_{i=1}^m$ are known matrices, $\beta$ is an unknown $p \times
1$ vector of regression coefficients, $\{u_i\}_{i=1}^m$ are
independent random vectors whose elements represent the various levels
of the random factors in the model, and $e \sim \mbox{N}_n(0,
\lambda^{-1}_0 I)$.  The random vectors $e$ and $u := \big( u_1^T \;
u_2^T \cdots u_m^T \big)^T$ are independent, and $u \sim \mbox{N}_q(0,
\Lambda^{-1})$, where $u_i$ is $q_i \times 1$, $q = q_1 + \cdots +
q_m$, and $\Lambda = \oplus_{i=1}^m \lambda_{i} I_{q_i}$.  (We assume
throughout that $n \ge 2$, and that $q_i \ge 2$ for each
$i=1,2,\dots,m$.)  For a book-length treatment of this model and its
many applications, see \citet{mccu:sear:neuh:2008}.

In the Bayesian setting, prior distributions are assigned to $\beta$
and $\lambda := (\lambda_0 \;\, \lambda_{1} \, \cdots \,
\lambda_{m})^T$.  Unfortunately, any non-trivial prior leads to an
intractable posterior density.  However, if $\beta$ and $\lambda$ are
assigned conditionally conjugate priors, then a simple two-block Gibbs
sampler can be used to explore the resulting posterior density.  In
particular, if we assign a multivariate normal prior to $\beta$, and
independent gamma priors to the precision parameters, then, letting
$\theta = (\beta^{T} \ u^{T})^{T}$, it is easily shown that given observed data $y$,
$\theta|\lambda,y$ is multivariate normal, and $\lambda|\theta,y$ is a
product of independent gammas.  (Since $u$ is unobservable, it is
treated like a parameter.)  Convergence rate results for this block
Gibbs sampler can be found in \citet{abra:hobe:2017}.

Now consider this Bayesian mixed model in the high-dimensional setting
where $p \gg n$.  This situation can arise, e.g., in genetics and
neuroscience where variability between subjects is most appropriately
handled with random effects \citep[see,
e.g.,][]{fazli2011,rohart2013fixed}.  While the model described above
could certainly be used in this setting, the multivariate normal prior
on $\beta$ is really not suitable.  Indeed, when $p \gg n$, it is
often assumed that $\beta$ is sparse, i.e., that many components of
$\beta$ are zero.  Unfortunately, the multivariate normal prior for
$\beta$ will \textit{shrink} the estimated coefficients towards zero,
but not enough to produce an (approximately) sparse estimate of $\beta$.
Additionally, when the components of $\beta$ have
varying magnitudes, the estimates of the ``large'' components will be
shrunk disproportionately compared to the estimates of the ``small''
components.  Below we propose an alternative prior for $\beta$ that is
tailored to the high-dimensional setting.

The well-known Bayesian interpretation of the lasso (involving iid
Laplace priors for the regression parameters) has led to a flurry of
recent research concerning the development of prior distributions for
regression parameters (in linear models \textit{without} random
effects) that yield posterior distributions with high posterior
probability around sparse values of $\beta$.  These prior
distributions are called {\it continuous shrinkage priors} and the
corresponding statistical models are referred to as {\it Bayesian
  shrinkage models} (see, e.g.,
\cite{bhattacharya2013bayesian,bhattacharya2015dirichlet},
\cite{griffin2010inference}, \cite{polson2010shrink}, and
\cite{park2008bayesian}).  One such Bayesian shrinkage model is the
so-called normal-gamma model of \cite{griffin2010inference}, which is
given by
\begin{equation}
\label{eq:normalgamma}
\begin{aligned}
  Y| \beta,\tau,\lambda_0 & \sim \mbox{N}_{n}(X\beta,\lambda_0^{-1}I_{n}) \\
  \beta | \tau,\lambda_0 & \sim \mbox{N}_{p}(0,\lambda_0^{-1}
  D_{\tau}) \;,
\end{aligned}
\end{equation}
where $\tau := (\tau_{1} \, \cdots \, \tau_{p})^{T}$ and $D_{\tau}$ is
a diagonal matrix with the $\tau_i$s on the diagonal.  The precision
parameter, $\lambda_0$, and the components of $\tau$ are assumed to be
\textit{a priori} independent with $\lambda_0 \sim \text{Gamma}(a,b)$ and
$\tau_{i} \iid \text{Gamma}(c,d)$ for $i=1,\ldots,p$.  When $c=1$,
this model becomes the Bayesian lasso model introduced by
\cite{park2008bayesian}.  We note that
\cite{bhattacharya2013bayesian,bhattacharya2015dirichlet} show that,
in terms of frequentist optimality, the Bayesian lasso has sub-optimal
prior concentration rates in that it does not place sufficient mass
around sparse values of $\beta$.  Alternatively, shrinkage priors that
have singularities at zero and robust tails (such as in the
normal-gamma model with $c < 1/2$), have been shown to perform well in
empirical studies.

In this paper, we propose and analyze an MCMC algorithm for a new
Bayesian general linear mixed model in which the standard multivariate
normal prior on $\beta$ is replaced with the continuous shrinkage
prior from the normal-gamma model.  Our high-dimensional Bayesian
general linear mixed model is defined as follows
\begin{equation}
\begin{aligned}
\label{eq:NormalGammaModel}
Y|\beta,u,\tau,\lambda & \sim \mbox{N}_{n}\left(X\beta + \sum_{i=1}^{m}Z_{i}u_{i}, \lambda_{0}^{-1}I_{n}\right) \\
\beta|u,\tau,\lambda & \sim \mbox{N}_{p}(0,\lambda_{0}^{-1}D_{\tau}) \\
u|\tau,\lambda & \sim \mbox{N}_{q}(0, \Lambda^{-1}) \;,\\
\end{aligned}
\end{equation}
where $\lambda$ and $\tau$ are a priori independent with $\lambda_{i}
\ind \text{Gamma}(a_{i},b_{i})$, for $i=0,1,\ldots,m$, and $\tau_{i}
\iid \text{Gamma}(c,d)$ for $i=1,\ldots,p$.  This model can be
considered a Bayesian analog of the frequentist, high dimensional
mixed model developed by \cite{schelldorfer2011estimation}.  (Of
course, it can also be viewed as a mixed version of the normal-gamma
shrinkage model.) A similar sparse Bayesian linear mixed model has been
proposed by \cite{zhou2013polygenic} for polygenic modeling. They
assume a ``spike and slab" prior consisting of a mixture of a point mass at 0
and a normal distribution for the components of $\beta$.
However, it is well-known that spike and slab priors lead to MCMC algorithms that
have convergence problems, especially when $p$ is large
(\cite{polson2010shrink,bhattacharya2015dirichlet}).

Recall that $\theta = (\beta^{T} \ u^{T})^{T}$, and let $\pi(\theta,\lambda,\tau|y)$ denote the posterior density
associated with model \eqref{eq:NormalGammaModel}.  This density is
highly intractable and Bayesian inference requires MCMC, which should,
of course, be based on a geometrically ergodic Monte Carlo Markov
chain \citep[see,
e.g.][]{robe:rose:1998,jone:hobe:2001,fleg:hara:jone:2008}.  As we
show in Section~\ref{sec:hybrid_samp}, the full conditional densities $\pi_1(\theta|\lambda,\tau,y)$,
$\pi_2(\lambda|\theta,\tau,y)$, and $\pi_3(\tau|\theta,\lambda,y)$ all
have standard forms, which means that there is a simple three-block
Gibbs sampler available.  Unfortunately, we have been unable to
establish a convergence rate for this Gibbs sampler (in either
deterministic or random scan form).  However, we have been able to
prove that a related \textit{hybrid} algorithm does converge at a
geometric rate.  The invariant density of our Markov chain is
$\pi(\theta,\lambda|y) := \int_{\mathbb{R}_+^p}
\pi(\theta,\lambda,\tau|y) \, d\tau$.  Let $r \in (0,1)$ be fixed, and
denote the Markov chain by $\{(\theta_k,\lambda_k)\}_{k=0}^\infty$.
If the current state is $(\theta_k,\lambda_k) = (\theta,\lambda)$,
then we simulate the new state, $(\theta_{k+1},\lambda_{k+1})$, using
the following three-step procedure.

\baro \vspace*{2mm}
\noindent {\rm Iteration $k+1$ of the hybrid algorithm:}
\begin{enumerate}
\item[1.] Draw $\tau \sim \pi_3(\cdot|\theta,\lambda,y)$, and,
  independently, $U \sim \mbox{Uniform}(0,1)$.
\item[2a.] If $U < r$, set $(\theta_{k+1},\lambda_{k+1}) =
  (\theta',\lambda)$ where $\theta' \sim \pi_1(\cdot|\lambda,\tau,y).$
\item[2b.] Otherwise, set $(\theta_{k+1},\lambda_{k+1}) =
  (\theta,\lambda')$ where $\lambda' \sim \pi_2(\cdot|\theta,\tau,y).$
\end{enumerate}
\barba
\medskip

At each iteration, this sampler first performs a deterministic update
of $\tau$ from its full conditional distribution.  Next, a random scan
update is performed, which updates either $\theta$ or $\lambda$ with
probability $r$ and $(1-r)$, respectively.  This sampler is no more
difficult to implement than the three-block Gibbs sampler.  Moreover,
it is straightforward to show that the Markov chain driving this
algorithm is reversible with respect to $\pi(\theta,\lambda|y)$,
and that it is Harris ergodic.  This algorithm is actually a special
case of a more general MCMC algorithm for Bayesian latent data models
that was recently developed by \cite{jung:2015}.

Our main result provides a set of conditions under which the hybrid
Markov chain defined above is geometrically ergodic (as defined by
\citet[][p.319]{jone:hobe:2001}).  Here is the result.
\begin{proposition}
  \label{prop:ge_p_pred}
  The Markov chain, $\{(\theta_{k},\lambda_{k})\}_{k=0}^{\infty}$, is
  geometrically ergodic for all $r \in (0,1)$ if
\begin{enumerate}
\item $Z := (Z_1 \; Z_2 \cdots Z_m)$ has full column rank,
\item $a_{0} > \frac{1}{2} \left( \rank(X)-n+(2c+1)p + 2 \right)$, and
\item $a_i > 1$ for each $i \in \{1,2,\dots,m\}$.
\end{enumerate}
\end{proposition}

Note that the conditions of Proposition~\ref{prop:ge_p_pred} are quite
simple to check.  We do require $Z$ to be full column rank, which
holds for most basic designs, but there is no such restriction on $X$,
so the result is applicable when $p>n$.  While the second condition
may become restrictive when $p \gg n$, this can be mitigated to some
extent by the fact that the user if free to choose any positive value
for the hyperparameter $b_{0}$.  Indeed, when a large value of $a_0$
is required to satisfy condition (2), $b_{0}$ can be chosen such that
the prior mean and variance of $\lambda_{0}$ (which are given by
$a_{0}b_{0}^{-1}$ and $a_{0}b_{0}^{-2}$, respectively) have reasonable
values.

\citet{abra:hobe:2017} established simple, easily-checked sufficient
conditions for geometric ergodicity of the two-block Gibbs sampler for
the standard Bayesian general linear mixed model that assigns a multivariate normal prior to
$\beta$.  Note, however, that unlike the multivariate normal prior,
the continuous shrinkage prior that we use here is hierarchical, which
means that an additional set of latent variables (the $\tau_i$s) must
be managed, and this leads to more complicated MCMC algorithms that
are more difficult to analyze.  On a related note,
\citet{pal2014geometric} obtained geometric ergodicity results for the
Gibbs sampler developed for the original normal-gamma model (without random
effects).  But again, adding random effects to the normal-gamma model leads to new MCMC algorithms
that are more complex and harder to handle.

The remainder of this paper is organized as follows.
Section~\ref{sec:hybrid_samp} contains a formal definition of the
Markov chain that drives our hybrid sampler.  A proof of
Proposition~\ref{prop:ge_p_pred} is given in
Section~\ref{sec:geom_erg}.  Finally, a good deal of technical
material, such as proofs of the lemmas that are used in the main
proof, has been relegated to the Appendices.

\section{The Hybrid Sampler}
\label{sec:hybrid_samp}

In this section, we formally define the Markov transition function
(Mtf) of the hybrid algorithm.  We begin with a brief derivation of
the conditional densities, $\pi_i$, $i=1,2,3$.  Let $W = [X \; Z]$,
$\mathbb{R}_+ = (0,\infty)$, and recall that $\theta= (\beta^{T} \
u^{T})^{T}$.  The posterior density can be expressed up to a constant
of proportionality as
\begin{equation}
\begin{aligned}
\label{eq:fullposterior}
\pi(\theta, \tau,\lambda| y) & \propto \, \pi(y|\beta,u,\tau,\lambda)\pi(\beta|\tau,\lambda)\pi(u|\tau,\lambda)\pi(\tau) \pi(\lambda) \\
& \propto \,\lambda_{0}^{n/2}\exp\left\{-\frac{\lambda_{0}}{2}(y-W\theta)^{T}(y-W\theta)\right\} \\
& \quad \times \lambda_{0}^{p/2}\left[\prod_{j=1}^{p}\tau_{j}^{-1/2}\right]\exp\left\{-\frac{\lambda_{0}}{2}\beta^{T}D_{\tau}^{-1}\beta \right\}\\
& \quad\times\left[\prod_{i=1}^{m}\lambda_{i}^{q_{i}/2}\right]\exp\left\{-\frac{1}{2}u^{T}\Lambda u\right\} \\
& \quad \times
\left[\prod_{j=1}^{p}\tau_{j}^{c-1}e^{-d\tau_{j}}\,I_{\mathbb{R}_+}(\tau_{j})\right]\left[\prod_{i=0}^{m}\lambda_{i}^{a_{i}-1}e^{-b_{i}\lambda_{i}}\,I_{\mathbb{R}_+}(\lambda_{i})\right]
.
\end{aligned}
\end{equation}
We will use \eqref{eq:fullposterior} to derive the full conditional
distributions of $\theta$, $\tau$ and $\lambda$.  First, it is shown
in the Appendix that the full conditional distribution of $\theta$ is
multivariate normal with
\begin{equation}
\label{eq:theta_cond_mean}
\Ex[\theta|\tau,\lambda,y] = \left[\begin{array}{c}
    \lambda_{0} \tsli X^{T}y - \lambda_{0}^{2}\tsli X^{T}Z\qsli Z^{T}\msl y\\
    \lambda_{0}\qsli Z^{T}\msl y \;
\end{array}\right] ,
\end{equation}
and
\begin{equation}
\label{eq:theta_cond_var}
\Var[\theta|\tau,\lambda,y] = \left[\begin{array}{cc}
    \tsl^{-1} + \lambda_{0}^{2}\tsl^{-1}X^{T}Z\qsl^{-1}Z^{T}X\tsl^{-1} & -\lambda_{0}\tsl^{-1}X^{T}Z\qsl^{-1}\\
    -\lambda_{0}\qsl^{-1}Z^{T}X\tsl^{-1} & \qsl^{-1}
  \end{array}\right] ,
\end{equation}
where $\tsl = \lambda_{0}(X^{T}X + D_{\tau}^{-1})$, $\msl = I
-\lambda_{0}XT^{-1}_{\lambda,\tau}X^{T}$, and $\qsl=
\lambda_{0}Z^{T}M_{\lambda,\tau}Z + \Lambda$ .

Next, it's clear from \eqref{eq:fullposterior} that the components of
$\lambda$ are conditionally independent, and that each has a gamma
distribution.  Indeed,
\begin{equation}
  \label{eq:lambda0_post}
  \pi_2(\lambda_{0}|\theta,\tau,y) \propto \, \lambda_{0}^{n/2 + p/2 + a_{0}-1} e^{-\lambda_{0}\left(\frac{||y-W\theta||^{2}}{2} + \frac{\beta^{T}D_{\tau}^{-1}\beta}{2} + b_{0}\right)}\,I_{\mathbb{R}_+}(\lambda_{0}) \;,
\end{equation}
and, for $i=1,2,\dots,m$,
\begin{equation}
\label{eq:lambda_i_post}
\pi_2(\lambda_{i}|\theta,\tau,y) \propto \, \lambda_{i}^{q_{i}/2 + a_{i} - 1}e^{-\lambda_{i}\left(\frac{||u_{i}||^{2}}{2}+b_{i}\right)}\,I_{\mathbb{R}_+}(\lambda_{i}) \;.
\end{equation}

Lastly,
\begin{align*}{}
  \pi_3(\tau|\theta,\lambda,y) & \propto \,\left[\prod_{j=1}^{p}\tau_{j}^{-1/2}\right]\exp\left\{-\frac{\lambda_{0}}{2}\sum_{j=1}^{p}\frac{\beta_{j}^{2}}{\tau_{j}}\right\}\left[\prod_{j=1}^{p}\tau_{j}^{c-1}e^{-d\tau_{j}}\,I_{\mathbb{R}_+}(\tau_{j})\right] \\
  &=
  \prod_{j=1}^{p}\tau_{j}^{(c-1/2)-1}\exp\left\{-\frac{1}{2}\left(\frac{\lambda_{0}\beta_{j}^{2}}{\tau_{j}}
      + 2d\tau_{j}\right)\right\}\,I_{\mathbb{R}_+}(\tau_{j}) \;.
\end{align*}
Thus, the $\tau_j$s are conditionally independent, and
\begin{equation}
\label{eq:tau_post}
\pi_3(\tau_{j}|\theta,\lambda,y) \propto \, \tau_{j}^{(c -
  1/2)-1}e^{-\frac{1}{2}\left(\lambda_{0}\frac{\beta_{j}^{2}}{\tau{j}}
    + 2d\tau_{j}\right)}\, I_{\mathbb{R}_+}(\tau_{j}) \;.
\end{equation}
This brings us to a subtle technical problem that turns out to be very
important in our convergence rate analysis.  Note that, if $c \le 1/2$
and $\beta_j=0$, then the right-hand side of \eqref{eq:tau_post} is
not integrable.  Moreover, it is precisely these values of $c$ that yield
effective shrinkage priors.  Of course, from a
simulation standpoint, this technical problem is a non-issue because
we will never observe an exact zero from the normal distribution.
However, in order to perform a theoretical analysis of the Markov
chain, we are obliged to define the Mtf for \textit{all} points in the
state space.  Our solution is to simply delete the offending points
from the state space.  (Alternatively, we could make a special
definition of the Mtf at the offending points, but this leads to a
Markov chain that lacks the \textit{Feller} property
\citep[][p.124]{meyn:twee:2009}, and this prevents us from employing
\pcite{meyn:twee:2009} Lemma 15.2.8.)  Thus, we define the state space
of our Markov chain to be
\[
\X = \Big \{ (\theta,\lambda) \in \mathbb{R}^{p+q} \times
\mathbb{R}_+^{m+1} : |\beta_j|>0 \;\; \mbox{for $j=1,\dots,p$} \Big \}
\;.
\]
Taking the state space to be $\X$ instead of $\mathbb{R}^{p+q} \times
\mathbb{R}_+^{m+1}$ has no effect on posterior inference because the
difference between these two sets is a set of measure zero.  However,
as will become clear in Section~\ref{sec:geom_erg}, the deleted points
do create some complications in the drift analaysis.  We note that
this particular technical issue has surfaced before \citep[see,
e.g.,][]{roman2012convergence}, although, in contrast with the current
situation, the culprit is typically improper priors.

It's clear from \eqref{eq:tau_post} that conditional on $\theta$,
$\lambda$ and $y$, the distribution of $\tau_{j}$ is
$\mbox{GIG}(c-1/2,2d,\lambda_{0}\beta_{j}^{2})$, where
$\mbox{GIG}(\zeta,\xi,\psi)$ denotes the Generalized Inverse Gaussian
distribution with parameters $\zeta \in \mathbb{R}$, $\xi > 0$, and
$\psi > 0$.  The density is given by
\begin{equation}
  \label{eq:GIG_dens}
  f_{\mbox{GIG}}(x;\zeta,\xi,\psi) =
  \frac{\xi^{\zeta/2}}{2\psi^{\zeta/2} K_{\zeta}(\sqrt{\xi\psi})}
  x^{\zeta-1} e^{-\frac{1}{2}(\xi x + \frac{\psi}{x})}
  I_{\mathbb{R}_{+}}(x) \;,
\end{equation}
where $K_{\zeta}(\cdot)$ denotes the modified Bessel function of the
second kind.

Fix $r \in (0,1)$, and let $A$ denote a measurable set in $\X$.  The
Mtf of the Markov chain that drives our hybrid algorithm is given by
\begin{align*}
  P((\theta,\lambda),A) & = r \int_{\mathbb{R}^{p+q}}
  I_{A}(\tilde{\theta},\lambda) \left[ \int_{\mathbb{R}^{p}_{+}}
    \pi(\tilde{\theta} | \tau,\lambda) \, \pi(\tau|\theta,\lambda) \,
    d\tau \right] \, d\tilde{\theta} \\ & \quad \quad +(1-r)
  \int_{\mathbb{R}^{m}_{+}} I_{A}(\theta,\tilde{\lambda}) \left[
    \int_{\mathbb{R}^{p}_{+}} \pi(\tilde{\lambda} | \theta,\tau) \,
    \pi(\tau|\theta,\lambda) \, d\tau \right] \, d\tilde{\lambda} \;,
\end{align*}
where we (henceforth) omit the dependence on $y$ from the conditional
densities for notational convenience.  The Appendix contains a proof
that the Markov chain defined by $P$ is a Feller chain.  In the next
section, we prove Proposition~\ref{prop:ge_p_pred}.

\section{Proof of Proposition~\ref{prop:ge_p_pred}}
\label{sec:geom_erg}

This section contains a proof of Proposition~\ref{prop:ge_p_pred}.  In
particular, we establish a geometric drift condition for our Markov
chain, $\{(\theta_k,\lambda_k)\}_{k=0}^\infty$, using a drift
function, $v: \X \rightarrow [0,\infty)$, that is unbounded off
compact sets.  Since our chain is Feller, geometric ergodicity then
follows immediately from \pcite{meyn:twee:2009} Theorem 6.0.1 and
Lemma 15.2.8.

\subsection{The drift function}
\label{sec:drift_fuc}

Recall that in our model, $\tau_{i} \iid \text{Gamma}(c,d)$.  The
hyperparameter $c$ will play a crucial role in our drift function.
Define $\nu: \mathbb{R}_{+} \rightarrow (0,1/2]$ as
\[
\nu(c) = c \, I_{(0,\frac{1}{2}]}(c) + \min\{1/2,2c-1\} \,
I_{(\frac{1}{2}, \infty)}(c) \;.
\]
Now let $\delta = (\delta_{1} \cdots \delta_{m})^{T}$, $\eta =
(\eta_{1} \cdots \eta_{m})^{T}$, and define the drift function as
follows
\begin{equation}
  \label{eq:drift_fun}
\begin{aligned}
  v(\theta,\lambda) & = \alpha_{1} \norm{y-W\theta}^{2} + \alpha_{2}\norm{\beta}^{2}\\
  & \quad+ \sum_{j=1}^{p} \frac{1}{\abs{\beta_{j}}^{\nu(c)}}
  + \sum_{i=1}^{m} \delta_{i} \norm{u_{i}}^{2} \\
  & \quad + \alpha_{3} \lambda_{0} + \alpha_{4}
  \lambda_{0}^{\nu(c)/2} + \lambda_{0}^{-1} + \sum_{i=1}^{m}
  \lambda_{i} + \sum_{i=1}^{m}\eta_{i}\lambda_{i}^{-1} \;,
\end{aligned}
\end{equation}
where $\alpha_{1},\alpha_{2},\alpha_{3},\alpha_{4} \in \mathbb{R}_+$
and $\delta, \eta \in \mathbb{R}^{m}_{+}$.  The values of these
constants are to be determined.  The Appendix contains a proof that
$v(\theta,\lambda)$ is unbounded off compact sets.

\begin{remark}
  Using $\X$ instead of $\mathbb{R}^{p+q} \times \mathbb{R}_+^{m+1}$
  as the state space has implications for the construction of the
  drift function.  In particular, since the hyper-planes where
  $\beta_{j} = 0$ are not part of $\X$, and we need the drift function
  to be unbounded off compact sets, we must have a term in the drift
  function that diverges as $\abs{\beta_{j}} \rightarrow 0$.  In fact,
  this is the \textit{only} reason why the term $\sum_{j=1}^{p}
  \frac{1}{\abs{\beta_{j}}^{\nu(c)}}$ is part of $v$.  Interestingly,
  \citet{pal2014geometric} established their convergence rate results
  using a drift/minorization argument, which does not require the
  drift function to be unbounded off compact sets, yet they still require
  this same term in their drift function.  Fortunately, we are able to
  reuse some of the bounds that they developed.
\end{remark}

Our goal is to demonstrate that
\begin{equation}
  \label{eq:driftcond}
  \Ex[v(\tilde{\theta},\tilde{\lambda})|\theta,\lambda] \leq \rho \,
  v(\theta,\lambda) + L \;,
\end{equation}
for some $\rho \in [0,1)$ and some finite constant $L$.  First, note
that
\begin{align*}
  \Ex[v(\tilde{\theta},\tilde{\lambda})|\theta,\lambda] &= r\int_{\mathbb{R}^{p+q}} v(\tilde{\theta},\lambda)\left[\int_{\mathbb{R}^{p}_{+}}\pi(\tilde{\theta}|\tau,\lambda)\pi(\tau|\theta,\lambda)d\tau\right]\, d\tilde{\theta} \\
  & \hspace*{10mm} + (1-r)\int_{\mathbb{R}_{+}^{m}}v(\theta,\tilde{\lambda})\left[\int_{\mathbb{R}^{p}_{+}}\pi(\tilde{\lambda}|\theta,\tau)\pi(\tau|\theta,\lambda)\, d\tau\right]\,d\tilde{\lambda}\\
  &=
  r\Ex\left[\Ex[v(\tilde{\theta},\lambda)|\tau,\lambda]\bigg|\theta,\lambda\right]
  + (1-r)\Ex\left[\Ex[v(\theta,\tilde{\lambda})|\tau,\theta]\bigg|
    \theta,\lambda\right] \;.
\end{align*}
It follows that
\begin{equation}
\label{eq:main_drift_eq}
\begin{aligned}
  \Ex[v(\tilde{\theta},\tilde{\lambda})|\theta,\lambda] & = r\Ex\left[\Ex\left[\alpha_{1}\norm{y-W\tilde{\theta}}^{2} + \alpha_{2}\norm{\tilde{\beta}}^{2} \right.\right. \\
  &\hspace*{25mm} \left.\left.+ \sum_{j=1}^{p}\frac{1}{\abs{\tilde{\beta}_{j}}^{\nu(c)}}+ \sum_{i=1}^{m}\delta_{i}\norm{\tilde{u}_{i}}^{2}\Big|\tau,\lambda\right]\bigg|\theta,\lambda\right] \\
  &\quad + r\left(\alpha_{3}\lambda_{0} +\alpha_{4}\lambda_{0}^{\nu(c)/2}+ \lambda_{0}^{-1} + \sum_{i=1}^{m}\lambda_{i} + \sum_{i=1}^{m}\eta_{i}\lambda_{i}^{-1}\right)\\
  & \quad + (1-r)\bigg(\alpha_{1}\norm{y-W\theta}^{2} + \alpha_{2}\norm{\beta}^{2}  + \sum_{j=1}^{p}\frac{1}{\abs{\beta_{j}}^{\nu(c)}}+ \sum_{i=1}^{m}\delta_{i}\norm{u_{i}}^{2}\bigg)\\
  &\quad +(1-r)\Ex\bigg[\Ex\bigg[\alpha_{3}\tilde{\lambda}_{0}
      +\alpha_{4}\tilde{\lambda}_{0}^{\nu(c)/2} +
      \tilde{\lambda}_{0}^{-1} + \sum_{i=1}^{m}\tilde{\lambda}_{i} +
      \sum_{i=1}^{m}\eta_{i}\tilde{\lambda}_{i}^{-1}\Big|\tau,\theta\bigg]\bigg|\theta,\lambda\bigg]
  \;.
\end{aligned}
\end{equation}

\begin{remark}
  The hybrid sampler employs three full conditional densities, yet the
  expression above contains only \textit{two} nested expectations.
  This is due to the random scan step in the hybrid sampler.  In
  contrast, a drift analysis of the deterministic scan Gibbs sampler
  for this problem involves similar equations having \textit{three}
  nested expectations, which greatly complicates the calculations.  On
  the other hand, a drift analysis of the random scan Gibbs sampler
  involves no nested expectations, which suggests that such an
  analysis would be relatively easy.  However, it turns out to be more
  difficult than the analysis of the hybrid algorithm.
\end{remark}

Now define
\[
h(\tilde{\theta}) = \alpha_{1}\norm{y-W\tilde{\theta}}^{2} +
\alpha_{2}\norm{\tilde{\beta}}^{2} +
\sum_{j=1}^{p}\frac{1}{\abs{\tilde{\beta}_{j}}^{\nu(c)}}+
\sum_{i=1}^{m}\delta_{i}\norm{\tilde{u}_{i}}^{2} \;,
\]
and
\[
g(\tilde{\lambda}) = \alpha_{3}\tilde{\lambda}_{0} +
\alpha_{4}\tilde{\lambda}_{0}^{\nu(c)/2} + \tilde{\lambda}_{0}^{-1} +
\sum_{i=1}^{m}\tilde{\lambda}_{i} +
\sum_{i=1}^{m}\eta_{i}\tilde{\lambda}_{i}^{-1} \;.
\]
In order to bound \eqref{eq:main_drift_eq}, we need to develop bounds
for $\Ex \Big[ \Ex \big[ h(\tilde{\theta}) \big| \tau, \lambda \big]
\Big| \theta,\lambda \Big]$ and $\Ex \Big[ \Ex \big[
g(\tilde{\lambda}) \big| \tau, \theta \big] \Big| \theta, \lambda
\Big]$.  These will be handled separately in the next two subsections.

\subsection{A bound on $\Ex \Big[ \Ex \big[ h(\tilde{\theta}) \big|
  \tau, \lambda \big] \Big| \theta,\lambda \Big]$}
\label{sec:first_term}

Let
$$R_i = [0_{q_i \times q_1} \ldots 0_{q_i \times q_{i-1}} \
I_{q_i\times q_i } \ 0_{q_i\times q_{i+1}}\ldots 0_{q_ i \times q_r}],$$
so that $u_i = R_i u$ for $i=1,\ldots,m$.  It's easy to see that
\begin{equation}
\label{eq:theta_exp}
\begin{aligned}
  \Ex\left[\norm{y-W\tilde{\theta}}^{2}\Big|\tau,\lambda\right] & = \tr(W\Var[\tilde{\theta}|\tau,\lambda]W^{T}) + \norm{y-W\Ex[\tilde{\theta}|\tau,\lambda]}^{2} \;, \\
  \Ex\left[\norm{\tilde{\beta}}^{2}\Big|\tau,\lambda\right] & = \tr(\Var[\tilde{\beta}|\tau,\lambda]) + \norm{\Ex[\tilde{\beta}|\tau,\lambda]}^{2} \;, \hspace*{1mm} \mbox{and} \\
  \Ex\left[\norm{\tilde{u}_{i}}^{2}\Big|\tau,\lambda\right] & = \tr(R_{i}\qsli R_{i}^{T}) + \norm{\Ex[\tilde{u}_{i}|\tau,\lambda]}^{2} \;. \\
\end{aligned}
\end{equation}
Hence,
\begin{equation}
\label{eq:term1_eq1}
\begin{aligned}
  \Ex\left[\Ex\left[h(\tilde{\theta})\Big|\tau,\lambda\right]\bigg|\theta,\lambda\right]
  & = \alpha_{1}\Ex\bigg[\tr(W\Var[\tilde{\theta}|\tau,\lambda]W^{T}) + \norm{y-W\Ex[\tilde{\theta}|\tau,\lambda]}^{2}\big|\theta,\lambda\bigg] \\
  & \quad + \alpha_{2}\Ex\left[\tr(\Var[\tilde{\beta}|\tau,\lambda]) + \norm{\Ex[\tilde{\beta}|\tau,\lambda]}^{2}\big|\theta,\lambda\right] \\
  & \quad + \Ex\left[\Ex\left[\sum_{j=1}^{p}\frac{1}{\abs{\tilde{\beta}_{j}}^{\nu(c)}}\bigg|\tau,\lambda\right]\bigg|\theta,\lambda\right]\\
  & \quad + \sum_{i=1}^{m}\delta_{i}\Ex\big[\tr(R_{i}\qsli
    R_{i}^{T})+\norm{\Ex[\tilde{u}_{i}|\tau,\lambda]}^{2}\big|\theta,\lambda\big]\;.
\end{aligned}
\end{equation}
We will bound \eqref{eq:term1_eq1} using the following lemmas, which
are proven in the Appendix.
\begin{lemma}
\label{lemma:theta_tr_ineq}
For all $\tau \in \mathbb{R}^{p}_{+}$ and $\lambda \in
\mathbb{R}^{m+1}_{+}$ ,
\begin{itemize}
\item[(1)] $\tr(W\Var[\theta|\tau,\lambda]W^{T}) \leq \tr(X\tsli
  X^{T}) + \tr(Z\qsli Z^{T})$ , \hspace*{1mm} \mbox{and}

\item[(2)] $\tr(X\tsli X^{T}) \leq \rank(X)\lambda_{0}^{-1}$ ,

\item[(3)] $\tr(Z\qsli Z^{T}) \leq
  \tr(ZZ^{T})\sum_{i=1}^{m}\lambda_{i}^{-1}$ .
\end{itemize}
\end{lemma}

\begin{lemma}
\label{lemma:ynorm_ineq}
For all $\tau \in \mathbb{R}^{p}_{+}$ and $\lambda \in
\mathbb{R}^{m+1}_{+}$ ,
\[
\norm{y-W\Ex[\theta|\tau,\lambda]}^{2} \leq 2n\norm{y}^{2}+
2n^{3}\norm{y}^{2} \;.
\]
\end{lemma}

\begin{lemma}
\label{lemma:beta_ineq}
For all $\tau \in \mathbb{R}^{p}_{+}$ and $\lambda \in
\mathbb{R}^{m+1}_{+}$ ,
\begin{itemize}
\item[(1)] $\tr(\Var[\beta|\tau,\lambda]) \leq
  \lambda_{0}^{-1}\sum_{j=1}^{p}\tau_{j}
  +c^{*}\,\tr(ZZ^{T})\sum_{i=1}^{m}\lambda_{i}^{-1}$ , \hspace*{1mm}
  \mbox{and}

\item[(2)] $\norm{\Ex[\beta|\tau,\lambda]}^{2} \leq
  c^{*}n^{2}\norm{y}^{2}\left(s_{\max}^{2}\sum_{j=1}^{p}\tau_{j}+1\right)$
  .
\end{itemize}
where $s_{\max}$ is the largest singular value of $X$ and $c^{*}$ is a
finite positive constant.
\end{lemma}

\begin{lemma}
\label{lemma:frac_beta_lem}
For all $\tau\in\mathbb{R}^{p}_{+}$ and $\lambda \in
\mathbb{R}^{m+1}_{+}$ ,
\[
\Ex\left[\sum_{j=1}^{p}\frac{1}{\abs{\tilde{\beta}_{j}}^{\nu(c)}}\Big|\tau,\lambda\right]
\leq p\kappa(c)s_{\max}^{\nu(c)}\,\lambda_{0}^{\nu(c)/2} +
\kappa(c)\lambda_{0}^{\nu(c)/2}\sum_{j=1}^{p}\frac{1}{\tau_{j}^{\nu(c)/2}}
\;,
\]
where
\[
\kappa(c):=
\frac{\Gamma\left(\frac{1-\nu(c)}{2}\right)2^{\frac{1-\nu(c)}{2}}}{\sqrt{2\pi}}
\;,
\]
and $s_{\max}$ is the largest singular value of $X$.
\end{lemma}

\begin{lemma}
\label{lemma:u_ineq}
Assume that $Z$ has full column rank.  For all $\tau \in
\mathbb{R}^{p}_{+}$, $\lambda \in \mathbb{R}^{m+1}_{+}$ and
$i=1,\ldots,m$ ,
\begin{itemize}
\item[(1)] $\tr(R_{i}\qsli R_{i}^{T}) \leq q_{i}\lambda_{i}^{-1}$ , \hspace*{1mm}
  \mbox{and}

\item[(2)] $||\Ex[u_{i}|\tau,\lambda]||^{2} \leq
  q_{i}\tr[(Z^{T}Z)^{-1}]n^{3}\norm{y}^{2}\left(s_{\max}^{2}\sum_{j=1}^{p}\tau_{j}+1\right)$
  .
\end{itemize}
\end{lemma}
Substituting into (\ref{eq:term1_eq1}) gives
\begin{equation}
\label{eq:term1_eq2}
\begin{aligned}
  \Ex\left[\Ex\left[h(\tilde{\theta})\Big|\tau,\lambda\right]\right.&\bigg|\left.\theta,\lambda\right]\leq \alpha_{1}\rank(X)\lambda_{0}^{-1} + \alpha_{1}\tr(ZZ^{T})\sum_{i=1}^{m}\lambda_{i}^{-1}  \\
  & + \alpha_{2}\lambda_{0}^{-1}\sum_{j=1}^{p}\Ex\left[\tau_{j}\big|\theta,\lambda\right] + \alpha_{2}c^{*}\tr(ZZ^{T})\sum_{i=1}^{m}\lambda_{i}^{-1} \\
  & + \alpha_{2}c^{*}n^{2}\norm{y}^{2}s_{\max}^{2}\sum_{j=1}^{p}\Ex\left[\tau_{j}\big|\theta,\lambda\right] \\
  & + p\kappa(c)s_{\max}^{\nu(c)}\,\lambda_{0}^{\nu(c)/2} \\
  & + \kappa(c)\lambda_{0}^{\nu(c)/2}\sum_{j=1}^{p}\Ex\left[\frac{1}{\tau_{j}^{\nu(c)/2}}\bigg|\theta,\lambda\right]\\
  & + \sum_{i=1}^{m}\delta_{i}q_{i}\lambda_{i}^{-1}\\
  & + \max_{i}\{\delta_{i}\}q\,\tr[(Z^{T}Z)^{-1}]\norm{y}^{2}n^{3}s_{\max}^{2}\sum_{j=1}^{p}\Ex\left[\tau_{j}|\theta,\lambda\right]\\
  & + K_{0}(\alpha_{1},\alpha_{2},\delta) \;,
\end{aligned}
\end{equation}
where
\[
K_{0}(\alpha_{1},\alpha_{2},\delta) := 2\alpha_{1}(n\norm{y}^{2} +
n^{3}\norm{y}^{2}) +\alpha_{2}c^{*}n^{2}\norm{y}^{2} +
\tr[(Z^{T}Z)^{-1}]\norm{y}^{2}n^{3}\sum_{i=1}^{m}\delta_{i}q_{i} \;.
\]

\subsection{A bound on $\Ex \Big[ \Ex \big[ g(\tilde{\lambda}) \big|
  \tau, \theta \big] \Big| \theta, \lambda \Big]$}
\label{sec:second_term}

From (\ref{eq:lambda0_post}), we have
\begin{equation}
\label{eq:lambda0_exp1}
\begin{aligned}
\Ex[\lambda_{0}|\theta,\tau] &= \frac{\Gamma\left(\frac{n+p+2a_{0}+2}{2}\right)}{\Gamma\left(\frac{n+p+2a_{0}}{2}\right)}\left(\frac{||y-W\theta||^{2} + \beta^{T}D_{\tau}^{-1}\beta+2b_{0}}{2}\right)^{-1} \\
&\leq (n+p+2a_{0})\,b_{0}^{-1} \;,
\end{aligned}
\end{equation}
and
\begin{equation}
\label{eq:lambda0_exp2}
\begin{aligned}
  \Ex[\lambda_{0}^{-1}|\theta,\tau] &= \frac{\Gamma\left(\frac{n+p+2a_{0}-2}{2}\right)}{\Gamma\left(\frac{n+p+2a_{0}}{2}\right)}\left(\frac{\norm{y-W\theta}^{2} + \beta^{T}D_{\tau}^{-1}\beta+2b_{0}}{2}\right)\\
  &= \frac{1}{n+p+2a_{0}-2}\left(\norm{y-W\theta}^{2} +
    \beta^{T}D_{\tau}^{-1}\beta+2b_{0}\right) \;.
\end{aligned}
\end{equation}
Also, from Jensen's inequality and (\ref{eq:lambda0_exp1}),
\begin{equation}
\label{eq:lambda_nu_exp}
\Ex\left[\lambda_{0}^{\nu(c)/2}\big|\tau,\theta\right] \leq \Ex\left[\lambda_{0}|\tau,\theta\right]^{\nu(c)/2} \leq [(n+p+2a_{0})\,b_{0}^{-1}]^{\nu(c)} \;.
\end{equation}

Similarly, from (\ref{eq:lambda_i_post}), for $i = 1,2,\ldots,m$ we
have
\begin{equation}
\label{eq:lambda_i_exp1}
\Ex[\lambda_{i}|\theta,\tau] = \frac{\Gamma\left(\frac{q_{i}+2a_{i}+2}{2}\right)}{\Gamma\left(\frac{q_{i}+2a_{i}}{2}\right)}\left(\frac{\norm{u_{i}}^{2}+2b_{i}}{2}\right)^{-1} \leq (q_{i}+2a_{i})\,b_{i}^{-1} \;,
\end{equation}
and
\begin{equation}
\label{eq:lambda_i_exp2}
\begin{aligned}
  \Ex[\lambda_{i}^{-1}|\theta,\tau] &= \frac{\Gamma\left(\frac{q_{i}+2a_{i}-2}{2}\right)}{\Gamma\left(\frac{q_{i}+2a_{i}}{2}\right)}\left(\frac{\norm{u_{i}}^{2} +2b_{i}}{2}\right)\\
  &= \frac{1}{q_{i}+2a_{i}-2}\left(\norm{u_{i}}^{2} +2b_{i}\right) \;.
\end{aligned}
\end{equation}
Conditions (2) and (3) of Proposition \ref{prop:ge_p_pred} imply that
all of the above Gamma functions have positive arguments.

From (\ref{eq:lambda0_exp1})-(\ref{eq:lambda_i_exp2}), we have
\begin{equation}
\label{eq:term2_eq2}
\begin{aligned}
  \Ex\left[\Ex\left[g(\tilde{\lambda})\Big|\tau,\theta\right]\bigg|\theta,\lambda\right] &\leq \frac{1}{n+p+2a_{0}-2}\norm{y-W\theta}^{2}\\
  & \quad + \frac{1}{n+p+2a_{0}-2}\sum_{j=1}^{p}\beta_{j}^{2}\,\Ex\left[\tau_{j}^{-1}|\theta,\lambda\right]\\
  &\quad +
  \sum_{i=1}^{m}\frac{\eta_{i}}{q_{i}+2a_{i}-2}\norm{u_{i}}^{2} +
  K_{1}(\alpha_{3},\alpha_{4},\eta) \;,
\end{aligned}
\end{equation}
where
\begin{align*}
K_{1}(\alpha_{3},\alpha_{4},\eta) &:= \alpha_{3}(n+p+2a_{0}+1)\,b_{0}^{-1} \\
& \hspace*{15mm}+ \alpha_{4}[(n+p+2a_{0}+1)\,b_{0}^{-1}]^{\nu(c)}+ \frac{2 b_{0}}{n+p+2a_{0}-2}  \\
& \hspace*{15mm} + \sum_{i=1}^{m}(q_{i}+2a_{i}+1)\,b_{i}^{-1} + \sum_{i=1}^{m}\frac{2\eta_{i}b_{i}}{q_{i}+2a_{i}-2}.
\end{align*}
In the next subsection, we bound the right-hand sides of
\eqref{eq:term1_eq2} and \eqref{eq:term2_eq2}, and then combine these
new bounds to get a bound on
$\Ex[v(\tilde{\theta},\tilde{\lambda})|\theta,\lambda]$.

\subsection{A bound on
  $\Ex[v(\tilde{\theta},\tilde{\lambda})|\theta,\lambda]$}
\label{sec:bound}

If $X \sim \mbox{GIG}(\zeta,\xi,\psi)$, then
\begin{equation}
\label{eq:gig_exp}
\Ex[X^{p}] = \frac{\psi^{p/2}K_{\zeta+p}(\sqrt{\xi\psi})}{\xi^{p/2} K_{\zeta}(\sqrt{\xi\psi})} \;.
\end{equation}
Hence, for all $(\theta,\lambda) \in \mathsf{X}$,
\begin{equation}
\label{eq:tau_exp_ineq2}
\Ex[\tau_{j}|\theta,\lambda] = \sqrt{\frac{\lambda_{0}\beta_{j}^{2}}{2d}}\,\frac{K_{c+\frac{1}{2}}\left(\sqrt{2d\lambda_{0}\beta_{j}^{2}}\right)}{K_{c-\frac{1}{2}}\left(\sqrt{2d\lambda_{0}\beta_{j}^{2}}\right)} \;,
\end{equation}
\begin{equation}
\label{eq:tau_exp_ineq3}
\Ex[\tau_{j}^{-1}|\theta,\lambda] = \sqrt{\frac{2d}{\lambda_{0}\beta_{j}^{2}}}\,\frac{K_{c-\frac{3}{2}}\left(\sqrt{2d\lambda_{0}\beta_{j}^{2}}\right)}{K_{c-\frac{1}{2}}\left(\sqrt{2d\lambda_{0}\beta_{j}^{2}}\right)} \;,
\end{equation}
and
\begin{equation}
\label{eq:frac_tau_exp}
\Ex\left[\frac{1}{\tau_{j}^{\nu(c)/2}}\bigg|\theta,\lambda\right] = \left(\frac{2d}{\lambda_{0}\beta_{j}^{2}}\right)^{\nu(c)/4}\,\frac{K_{c-\frac{1}{2}-\frac{\nu(c)}{2}}\left(\sqrt{2d\lambda_{0}\beta_{j}^{2}}\right)}{K_{c-\frac{1}{2}}\left(\sqrt{2d\lambda_{0}\beta_{j}^{2}}\right)} \;,
\end{equation}
for $j=1,2,\ldots,p$.

Next, we will make use of the following lemmas, which are proved in
the Appendix.
\begin{lemma}
\label{lemma:tau_exp_lem}
For all $(\theta,\lambda) \in \mathsf{X}$ ,
\begin{itemize}
\item[1.] $\Ex[\tau_{j}|\theta,\lambda] \leq \frac{4c+1}{4d} +
  \frac{\lambda_{0}\beta_{j}^{2}}{2} \;, \hspace*{1mm} \mbox{and}$
\item[2.] $\Ex[\tau_{j}|\theta,\lambda] \leq
  \frac{c}{d}+\frac{\beta_{j}^{2}}{2C}+\frac{\lambda_{0}C}{4d} \;,$
\end{itemize}
for every $C > 0$.
\end{lemma}

\begin{lemma}
\label{lemma:tau_exp_lem2}
For all $(\theta,\lambda) \in \mathsf{X}$ ,
$$\Ex[\tau_{j}^{-1}|\theta,\lambda]  \leq d +\frac{3}{2\lambda_{0}\beta_{j}^{2}} \;.$$
\end{lemma}

\begin{lemma}
\label{lemma:frac_tau_lem}
For all $(\theta,\lambda) \in \mathsf{X}$ ,
$$\Ex\left[\frac{1}{\tau_{j}^{\nu(c)/2}}\Big| \theta,\lambda\right] \leq M_{1}\frac{1}{\lambda_{0}^{\nu(c)/2}\abs{\beta_{j}}^{\nu(c)}} + M_{2} \;,$$
where $M_{1}$ is a positive constant such that $M_{1}\kappa(c) < 1$,
and $M_{2}$ is a positive finite constant.
\end{lemma}

Applying Lemma~\ref{lemma:tau_exp_lem}, Lemma~\ref{lemma:frac_tau_lem}
and (\ref{eq:term1_eq2}), we have
\begin{align*}
\Ex\left[\Ex\left[h(\tilde{\theta})\Big|\tau,\lambda\right]\bigg|\theta,\lambda\right] &\leq \alpha_{1}\rank(X)\lambda_{0}^{-1} + \alpha_{1}\tr(ZZ^{T})\sum_{i=1}^{m}\lambda_{i}^{-1}  \\
&\quad  + \alpha_{2}\lambda_{0}^{-1}\sum_{j=1}^{p}\left(\frac{4c+1}{4d} + \frac{\lambda_{0}\beta_{j}^{2}}{2}\right) + \alpha_{2}c^{*}\tr(ZZ^{T})\sum_{i=1}^{m}\lambda_{i}^{-1} \\
&\quad + \alpha_{2}c^{*}n^{2}\norm{y}^{2}s_{\max}^{2}\sum_{j=1}^{p}\left(\frac{c}{d}+\frac{\beta_{j}^{2}}{2C_{1}}+\frac{\lambda_{0}C_{1}}{4d}\right) \\
&\quad  + p\kappa(c)s_{\max}^{\nu(c)}\,\lambda_{0}^{\nu(c)/2} \\
&\quad + \kappa(c)\lambda_{0}^{\nu(c)/2}\sum_{j=1}^{p}\left[M_{1}\frac{1}{\lambda_{0}^{\nu(c)/2}\abs{\beta_{j}}^{\nu(c)}} + M_{2}\right]\\
&\quad  + \sum_{i=1}^{m}\delta_{i}q_{i}\lambda_{i}^{-1}\\
& \quad  + \max_{i}\{\delta_{i}\}C_{3}\sum_{j=1}^{p}\left(\frac{c}{d}+\frac{\beta_{j}^{2}}{2C_{2}}+\frac{\lambda_{0}C_{2}}{4d}\right)\\
& \quad  +  K_{0}(\alpha_{1},\alpha_{2},\delta),
\end{align*}
where $C_{1}$ and $C_{2}$ are positive constants and $C_{3} = q\,\tr[(Z^{T}Z)^{-1}]\norm{y}^{2}n^{3}s_{\max}^{2}$. Hence,
\begin{equation}
\label{eq:term1_eq3}
\begin{aligned}
\Ex\left[\Ex\left[h(\tilde{\theta})\Big|\tau,\lambda\right]\bigg|\theta,\lambda\right] &\leq \alpha_{1}\rank(X)\lambda_{0}^{-1} + \alpha_{1}\tr(ZZ^{T})\sum_{i=1}^{m}\lambda_{i}^{-1}  \\
&\quad + \alpha_{2}\frac{p(4c+1)}{4d}\lambda_{0}^{-1}+\alpha_{2}\frac{\norm{\beta}^{2}}{2} \\
&\quad+ \alpha_{2}c^{*}\tr(ZZ^{T})\sum_{i=1}^{m}\lambda_{i}^{-1} \\
&\quad+ \alpha_{2}c^{*}n^{2}\norm{y}^{2}s_{\max}^{2}\left(\frac{\norm{\beta}^{2}}{2C_{1}} + \frac{pC_{1}}{4d}\lambda_{0}\right) \\
&\quad + p\kappa(c)\left(s_{\max}^{\nu(c)}+M_{2}\right)\lambda_{0}^{\nu(c)/2} \\
&\quad + \kappa(c)M_{1}\sum_{j=1}^{p}\frac{1}{\abs{\beta_{j}}^{\nu(c)}}+ \sum_{i=1}^{m}\delta_{i}q_{i}\lambda_{i}^{-1} \\
&\quad + \max_{i}\{\delta_{i}\}C_{3}\left(\frac{\norm{\beta}^{2}}{2C_{2}} + \frac{p C_{2}}{4d}\lambda_{0} \right)\\
& \quad +  K'_{0}(\alpha_{1},\alpha_{2},\delta),
\end{aligned}
\end{equation}
where
\begin{align*}
K'_{0}(\alpha_{1},\alpha_{2},\delta) &:= K_{0}(\alpha_{1},\alpha_{2},\delta) + \frac{pc}{d}\left(\alpha_{2}c^{*}n^{2}\norm{y}^{2}s_{\max}^{2} +C_{3}\max_{i}\{\delta_{i}\}\right).
\end{align*}
Next, using Lemma~\ref{lemma:tau_exp_lem2} and (\ref{eq:term2_eq2}),
we have
\begin{equation}
\label{eq:term2_eq3}
\begin{aligned}
  \Ex\left[\Ex\left[g(\tilde{\lambda})\Big|\tau,\theta\right]\bigg|\theta,\lambda\right]  &\leq \frac{1}{n+p+2a_{0}-2}\norm{y-W\theta}^{2} \\
  &\hspace*{10mm}+ \frac{1}{n+p+2a_{0}-2}\left(d\norm{\beta}^{2} +\frac{3p}{2\lambda_{0}}\right) \\
  & \hspace*{10mm} +
  \sum_{i=1}^{m}\frac{\eta_{i}}{q_{i}+2a_{i}-2}\norm{u_{i}}^{2} +
  K_{1}(\alpha_{3},\alpha_{4},\eta).
\end{aligned}
\end{equation}

Substituting (\ref{eq:term1_eq3}) and (\ref{eq:term2_eq3}) into
(\ref{eq:main_drift_eq}) and rearranging gives
\begin{equation}
\label{eq:drift_ineq10}
\begin{aligned}
  \Ex[v(\tilde{\theta},\tilde{\lambda})&|\theta,\lambda] \leq \alpha_{1}(1-r)\left[1 + \frac{1}{\alpha_{1}(n+p+2a_{0}-2)}\right]\norm{y-W\theta}^{2}\\
  &\quad + \alpha_{2}\left[\frac{r}{2} + \frac{rc^{*}n^{2}\norm{y}^{2}s^{2}_{\max}}{2C_{1}} + \frac{r\max_{i}\{\delta_{i}\}C_{3}}{2\alpha_{2}C_{2}}\right. \\
  &\hspace*{30mm}+ \left. (1-r) + \frac{(1-r)d}{\alpha_{2}(n+p+2a_{0}-2)}\right]\norm{\beta}^{2}\\
  & \quad + [1- r(1-\kappa(c)M_{1})]\sum_{j=1}^{p} \frac{1}{\abs{\beta_{j}}^{\nu(c)}} \\
  &\quad + \sum_{i=1}^{m}\delta_{i}(1-r)\left[1 + \frac{\eta_{i}}{\delta_{i}(q_{i}+2a_{i}-2)}\right]\norm{u_{i}}^{2}\\
  &\quad + \alpha_{3}\left[\frac{r\alpha_{2}c^{*}n^{2}\norm{y}^{2}s_{\max}^{2}pC_{1}}{4\alpha_{3}d} + \frac{r\max_{i}\{\delta_{i}\}pC_{2}C_{3}}{4\alpha_{3}d}+r\right]\lambda_{0} \\
  &\quad + \alpha_{4}r\left[1+\frac{p\kappa(c)\left(s_{\max}^{\nu(c)}+M_{2}\right)}{\alpha_{4}}\right]\lambda_{0}^{\nu(c)/2}\\
  &\quad+ \left[r\alpha_{1}\rank(X) + \frac{r\alpha_{2}p(4c+1)}{4d} + r \right.\\
  &\hspace*{40mm} \left.  + \frac{(1-r)\,3p}{2(n+p+2a_{0}-2)} \right]\lambda_{0}^{-1}\\
  &\quad + r\sum_{i=1}^{m}\lambda_{i} \\
  &\quad + \sum_{i=1}^{m}\eta_{i}\left[\frac{r\alpha_{1}\tr(ZZ^{T})}{\eta_{i}} + \frac{r\alpha_{2}c^{*}\tr(ZZ^{T})}{\eta_{i}}+\frac{r\delta_{i} q_{i}}{\eta_{i}} + r \right]\lambda_{i}^{-1}\\
  &\quad +
  L(\alpha_{1},\alpha_{2},\alpha_{3},\alpha_{4},\delta,\eta,r) \;,
\end{aligned}
\end{equation}
where
\[
L(\alpha_{1},\alpha_{2},\alpha_{3},\alpha_{4},\delta,\eta,r) :=
rK'_{0}(\alpha_{1},\alpha_{2},\delta) +
(1-r)K_{1}(\alpha_{3},\alpha_{4},\eta) \;.
\]

\subsection{The final step}
\label{sec:final}

Fix $r \in (0,1)$ and note that (aside from $L$) the terms of
\eqref{eq:drift_ineq10} agree with the terms of \eqref{eq:drift_fun},
except that each term in \eqref{eq:drift_ineq10} has an extra
\textit{constant} factor (coefficient).  Therefore, we can establish
that \eqref{eq:driftcond} holds by demonstrating the existence of
$\delta, \eta \in \mathbb{R}_{+}^{m}$ and $\alpha_{1}$, $\alpha_{2}$,
$\alpha_{3}$, $\alpha_{4}$, $C_{1}$, $C_{2} \in \mathbb{R}_+$ such that all
of these coefficients are simultaneously less than 1.  Moreover, if
the chain is geometrically ergodic for at least one $r \in (0,1)$ then
it is geometrically ergodic for all $r \in (0,1)$
\citep{jung:2015,jones2014convergence}.  Thus, we can treat $r$ as
another free parameter.  (A similar analysis was performed in
\citet{johnson2015geometric}.)

We begin by noting that two of the coefficients are always less than
1.  Indeed, the coefficient of $\sum_{i=1}^{m}\lambda_{i}$ is just
$r$, and the coefficient of $\sum_{j=1}^{p}
\frac{1}{\abs{\beta_{j}}^{\nu(c)}}$ is $[1- r(1-\kappa(c)M_{1})]$,
which is less than 1 since, by Lemma~\ref{lemma:frac_tau_lem}, $0<
\kappa(c)M_{1} < 1$.  Therefore, it suffices to show that we can
identify $\delta, \eta \in \mathbb{R}_{+}^{m}$ and $\alpha_{1}$,
$\alpha_{2}$, $\alpha_{3}$, $\alpha_{4}$, $C_{1}$, $C_{2} \in \mathbb{R}_+$
such that the following seven inequalities all hold simultaneously:

\begin{equation}
  \label{eq:rho_ineq1}
  \rho_{1}(\alpha_{1},r):= (1-r) \left[1 +
    \frac{1}{\alpha_{1}(n+p+2a_{0}-2)}\right] < 1 \;,
\end{equation}

\begin{equation}
\label{eq:rho_ineq2}
\begin{aligned}
  \rho_{2}(\alpha_{2},\delta,C_{1},C_{2},r)&:=\frac{r}{2} + \frac{rc^{*}n^{2}\norm{y}^{2}s^{2}_{\max}}{2C_{1}} + \frac{r\max_{i}\{\delta_{i}\}C_{3}}{2\alpha_{2}C_{2}}\\
  &\hspace*{20mm}+(1-r) + \frac{(1-r)d}{\alpha_{2}(n+p+2a_{0}-2)} < 1 \;,
\end{aligned}
\end{equation}

\begin{equation}
\label{eq:rho_ineq3}
\rho_{3i}(\delta_{i},\eta_{i},r):=(1-r)\left[1 + \frac{\eta_{i}}{\delta_{i}(q_{i}+2a_{i}-2)}\right] < 1 \;, \quad \text{for} \ i=1,\ldots,m \;,
\end{equation}

\begin{equation}
\label{eq:rho_ineq4}
\begin{aligned}
\rho_{4}(\alpha_{2},\alpha_{3},\delta,C_{1},C_{2},r) &:=\frac{r\alpha_{2}c^{*}n^{2}\norm{y}^{2}s_{\max}^{2}pC_{1}}{4\alpha_{3}d}\\
&\hspace*{20mm} + \frac{r\max_{i}\{\delta_{i}\}pC_{2}C_{3}}{4\alpha_{3}d}+r  < 1 \;,
\end{aligned}
\end{equation}

\begin{equation}
\label{eq:rho_ineq5}
\rho_{5}(\alpha_{4},r):= r\left[1+\frac{p\kappa(c)\left(s_{\max}^{\nu(c)}+M_{2}\right)}{\alpha_{4}}\right] < 1 \;,
\end{equation}

\begin{equation}
\label{eq:rho_ineq6}
\begin{aligned}
\rho_{6}(\alpha_{1},\alpha_{2},r)&:=r\alpha_{1}\rank(X) + \frac{r\alpha_{2}p(4c+1)}{4d}  \\
 &\hspace*{30mm} + r + \frac{(1-r)\,3p}{2(n+p+2a_{0}-2)}  < 1 \;,
 \end{aligned}
\end{equation}
and

\begin{equation}
\label{eq:rho_ineq7}
\begin{aligned}
\rho_{7i}(\alpha_{1},\alpha_{2},\delta_{i},\eta_{i},r)&:=\frac{r\alpha_{1}\tr(ZZ^{T})}{\eta_{i}} + \frac{r\alpha_{2}c^{*}\tr(ZZ^{T})}{\eta_{i}}\\
&\hspace*{20mm}+ \frac{r\delta_{i}q_{i}}{\eta_{i}} + r  < 1 \;, \quad \text{for} \ i=1,\ldots,m \;.
\end{aligned}
\end{equation}

We now derive a solution.  Solving (\ref{eq:rho_ineq1}) for
$\alpha_{1}$ gives
\begin{equation*}
  \alpha_{1} >\frac{(1-r)}{r(n+p+2a_{0}-2)} \;.
\end{equation*}
Next define
\begin{equation}
\label{eq:alpha_1_2}
\alpha_{1}^{*} = \frac{1}{r(n+p+2a_{0}-2)} \;,
\end{equation}
so that $\rho_{1}(\alpha_{1}^{*},r)<1$ for all $0<r<1$.  Next, let
\begin{equation}
\label{eq:alpha_2}
\alpha_{2}^{*} := \frac{2d}{r(n+p+2a_{0}-2)} \;.
\end{equation}
Substituting into (\ref{eq:rho_ineq2}) gives
\begin{equation}
\label{eq:alpha_2_1}
\begin{aligned}
  \rho_{2}(\alpha_{2}^{*},\delta,C_{1},C_{2},r) &= \frac{r}{2}  +  \frac{rc^{*}n^{2}\norm{y}^{2}s^{2}_{\max}}{2C_{1}} +\frac{r\max_{i}\{\delta_{i}\}C_{3}}{2\alpha_{2}^{*}C_{2}}+(1-r) + \frac{(1-r)d}{\alpha_{2}^{*}(n+p+2a_{0}-2)}\\
  &= \frac{r}{2} + \frac{rc^{*}n^{2}\norm{y}^{2}s^{2}_{\max}}{2C_{1}} +\frac{r^{2}\max_{i}\{\delta_{i}\}C_{3}(n+p+2a_{0}-2)}{4d C_{2}}\\
  &\quad\quad+(1-r) + (1-r)\frac{r}{2}\\
  &= \frac{rc^{*}n^{2}\norm{y}^{2}s^{2}_{\max}}{2C_{1}} +\frac{r^{2}\max_{i}\{\delta_{i}\}C_{3}(n+p+2a_{0}-2)}{4d C_{2}}+1-\frac{r^{2}}{2} \;.\\
\end{aligned}
\end{equation}
Thus, choosing
\begin{equation*}
  C_{1}^{*} > \frac{2c^{*}n^{2}\norm{y}^{2}s^{2}_{\max}}{r} \quad \text{and} \quad C_{2}^{*} > \frac{\max_{i}\{\delta_{i}\}C_{3}(n+p+2a_{0}-2)}{d} \;,
\end{equation*}
we get
\begin{equation}
\label{eq:alpha_2_2}
\rho_{2}(\alpha_{2}^{*},\delta,C_{1}^{*},C_{2}^{*},r) <\frac{r^{2}}{4} + \frac{r^{2}}{4} + 1- \frac{r^{2}}{2}  = 1 \;,
\end{equation}
for all $\delta\in\mathbb{R}_{+}^{m}$ and $0<r<1$.

Next, using (\ref{eq:rho_ineq6}), (\ref{eq:alpha_1_2}) and  (\ref{eq:alpha_2}), we get
\begin{equation}
\label{eq:rho5_eq1}
\begin{aligned}
  \rho_{6}(\alpha_{1}^{*},\alpha_{2}^{*},r) &=r\alpha_{1}^{*}\,\rank(X) + \frac{r\alpha_{2}^{*}\,p(4c+1)}{4d} + r + \frac{(1-r)3p}{2(n+p+2a_{0}-2)}\\
  &= \frac{\rank(X)}{n+p+2a_{0}-2} + \frac{p(4c+1)}{2(n+p+2a_{0}-2)} + r + \frac{(1-r)3p}{2(n+p+2a_{0}-2)} \\
  &< \frac{\rank(X)+ 2p(c+1)}{n+p+2a_{0}-2} + r \;,
\end{aligned}
\end{equation}
and from condition (2) of Proposition 1,
\[
0 < \frac{\rank(X)+ 2p(c+1)}{n+p+2a_{0}-2} < \frac{\rank(X)+
  2p(c+1)}{n+p+2\left(\frac{\rank(X)-n+(2c+1)p)+2}{2}\right)-2}=1
\;.
\]
Thus, for
\begin{equation}
  r < 1 - \frac{\rank(X)+ 2p(c+1)}{n+p+2a_{0}-2} \;,
\end{equation}
$\rho_{6}(\alpha_{1}^{*},\alpha_{2}^{*},r) < 1$.

Next, solving (\ref{eq:rho_ineq3}) for $\delta_{i}$ gives
\begin{equation}
  \delta_{i} > \frac{(1-r)\eta_{i}}{r(q_{i}+2a_{i} - 2)} \quad \text{for} \ i=1,\ldots,m \;.
\end{equation}
Hence, by defining
\begin{equation}
\label{eq:delta_i_ineq}
\delta_{i}^{*} = \frac{\eta_{i}}{r(q_{i}+2a_{i}-1)}, \quad  i=1,\ldots,m \;,
\end{equation}
it follows that $\rho_{3i}(\delta_{i}^{*},\eta_{i},r) < 1$ for all
$\eta_{i}>0$ and $r\in(0,1)$, $i=1,\ldots,m$.

Using equations (\ref{eq:rho_ineq7}), (\ref{eq:alpha_1_2}), (\ref{eq:alpha_2}) and (\ref{eq:delta_i_ineq}) we get
\begin{equation}
\label{eq:eta_ineq1}
\begin{aligned}
  \rho_{7i}(\alpha_{1}^{*},\alpha_{2}^{*},\delta_{i}^{*},\eta_{i},r) &= \frac{r\alpha_{1}^{*}\,\tr(ZZ^{T})}{\eta_{i}} + \frac{r\alpha_{2}^{*}\,c^{*}\tr(ZZ^{T})}{\eta_{i}}+ \frac{r\delta_{i}^{*}\,q_{i}}{\eta_{i}} + r\\
  &=\frac{\tr(ZZ^{T})}{\eta_{i}(n+p+2a_{0}-2)} + \frac{2d c^{*}\tr(ZZ^{T})}{\eta_{i}(n+p+2a_{0}-2)} + \frac{q_{i}}{q_{i} + 2a_{i}-2} + r \;,\\
\end{aligned}
\end{equation}
for $i=1,\ldots,m$, and from condition (3) of Proposition 1,
\begin{equation*}
  1-\frac{q_{i}}{q_{i} + 2a_{i}-2} = \frac{2a_{i}-2}{q_{i} + 2a_{i}-2} >0 \;.
\end{equation*}
Thus, for
\begin{equation}
\label{eq:eta_i_ineq1}
\eta_{i}^{*} > \frac{q_{i} + 2a_{i}-2}{2a_{i}-2}\left[\frac{\tr(ZZ^{T})}{(n+p+2a_{0}-2)} + \frac{2d c^{*}\tr(ZZ^{T})}{(n+p+2a_{0}-2)}\right] \;,
\end{equation}
it follows that
\begin{equation*}
  0 < \frac{\tr(ZZ^{T})}{\eta_{i}^{*}\,(n+p+2a_{0}-2)} + \frac{2d c^{*}\tr(ZZ^{T})}{\eta_{i}^{*}\,(n+p+2a_{0}-2)} + \frac{q_{i}}{q_{i} + 2a_{i}-2}< 1 \;,
\end{equation*}
for $i=1,\ldots,m$.  Hence,
$\rho_{7i}(\alpha_{1}^{*},\alpha_{2}^{*},\delta_{i}^{*},\eta_{i}^{*},r)
<1$ when
\begin{equation*}
  r < 1 - \frac{\tr(ZZ^{T})}{\eta_{i}^{*}(n+p+2a_{0}-2)} + \frac{2d c^{*}\tr(ZZ^{T})}{\eta_{i}^{*}(n+p+2a_{0}-2)} + \frac{q_{i}}{q_{i} + 2a_{i}-2} \;, \quad i=1,\ldots,m \;.
\end{equation*}

Next, solving (\ref{eq:rho_ineq4}) for $\alpha_{3}$ and (\ref{eq:rho_ineq5}) for $\alpha_{4}$ gives
\begin{equation}
\label{eq:alpha_4}
\alpha_{3} > \frac{1}{1-r}\left[\frac{r\alpha_{2}c^{*}n^{2}\norm{y}^{2}s_{\max}^{2}pC_{1}}{4d} + \frac{r\max_{i}\{\delta_{i}\}pC_{2}C_{3}}{4d}\right],
\end{equation}
and
\begin{equation}
\label{eq:alpha_5}
\alpha_{4} > \frac{r}{1-r}\left[p\kappa(c)\left(s_{\max}^{\nu(c)}+M_{2}\right)\right],
\end{equation}
respectively. Let $\alpha_{3}^{*}$ satisfy (\ref{eq:alpha_4}), then $\rho_{4}(\alpha_{2},\alpha_{3}^{*},\delta,C_{1},C_{2},r)<1$ for all $\alpha_{2},C_{1},C_{2}>0$, $\delta\in\mathbb{R}_{+}^{p}$ and $r\in(0,1)$. Now, let $\alpha_{4}^{*}$ satisfy (\ref{eq:alpha_5}), then $\rho_{5}(\alpha_{4}^{*},r) < 1$ for all $r\in(0,1).$

Lastly, choose $r^{*}$ such that
\begin{align*}
  r^{*} &< 1 - \max \left\{\frac{\rank(X)+ 2p(c+1)}{n+p+2a_{0}-2},\right.\\
  & \quad\quad\quad \left.\max_{1\leq i \leq
      m}\left\{\frac{\tr(ZZ^{T})}{\eta_{i}^{*}(n+p+2a_{0}-2)} +
      \frac{2d c^{*}\tr(ZZ^{T})}{\eta_{i}^{*}(n+p+2a_{0}-2)} +
      \frac{q_{i}}{q_{i} + 2a_{i}-2}\right\}\right\}.
\end{align*}
Let $\delta^{*}:=(\delta_{1}^{*} \ \delta_{2}^{*} \ \ldots \
\delta_{m}^{*})^{T}$ and $\eta^{*}:=(\eta_{1}^{*} \ \eta_{2}^{*} \
\ldots \ \eta_{m}^{*})^{T}$. The inequalities (\ref{eq:rho_ineq1}) -
(\ref{eq:rho_ineq7}) are then satisfied for $\delta^{*}$, $\eta^{*}$,
$\alpha_{1}^{*}$, $\alpha_{2}^{*}$, $\alpha_{3}^{*}$,
$\alpha_{4}^{*}$, $C_{1}^{*}$, $C_{2}^{*}$ and $r^{*}$.  Therefore,
\begin{equation}
\label{eq:drift_ineq11}
\begin{aligned}
  \Ex[v(\tilde{\theta},\tilde{\lambda})|\theta,\lambda] &\leq \rho_{1}(\alpha_{1}^{*},r^{*})\,\alpha_{1}^{*}\norm{y-W\theta}^{2}\\
  &\quad + \rho_{2}(\alpha_{2}^{*},\delta^{*},C_{1}^{*},C_{2}^{*},r^{*})\,\alpha_{2}^{*}\norm{\beta}^{2}\\
  &\quad + [1- r^{*}(1-\kappa(c)M_{1})]\sum_{j=1}^{p}\frac{1}{\abs{\beta_{j}}^{\nu(c)}} \\
  &\quad + \sum_{i=1}^{r}\rho_{3i}(\delta_{i}^{*},\eta_{i}^{*},r^{*})\,\delta_{i}^{*}\norm{u_{i}}^{2}\\
  &\quad + \rho_{4}(\alpha_{2}^{*},\alpha_{3}^{*},\delta^{*},C_{1}^{*},C_{2}^{*},r^{*})\,\alpha_{3}^{*}\lambda_{0} \\
  &\quad + \rho_{5}(\alpha_{4}^{*},r^{*})\,\alpha_{4}^{*}\lambda_{0}^{\nu(c)/2}\\
  &\quad + \rho_{6}(\alpha_{1}^{*},\alpha_{2}^{*},r^{*})\lambda_{0}^{-1}\\
  &\quad + r^{*}\sum_{i=1}^{r}\lambda_{i} + \sum_{i=1}^{r}\rho_{7i}(\alpha_{1}^{*},\alpha_{2}^{*},\delta_{i}^{*},\eta_{i}^{*})\,\eta_{i}^{*}\lambda_{i}^{-1}\\
  &\quad +
  L(\alpha_{1}^{*},\alpha_{2}^{*},\alpha_{3}^{*},\alpha_{4}^{*},\delta^{*},\eta^{*},r^{*})
  \;.
\end{aligned}
\end{equation}
To formally complete the argument, let $\rho^{*}$ denote the maximum
of all of the coefficients.  Then,
\begin{align*}
  \Ex[v(\tilde{\theta},\tilde{\lambda})|\theta,\lambda] &\leq \rho^{*}\left(\alpha_{1}^{*}\norm{y-W\theta}^{2} + \alpha_{2}^{*}\norm{\beta}^{2} + \sum_{j=1}^{p}\frac{1}{\abs{\beta_{j}}^{\nu(c)}} + \sum_{i=1}^{r}\delta_{i}^{*}\norm{u_{i}}^{2}\right. \\
  &\quad\quad\quad\quad\quad\quad \left. + \alpha_{3}^{*}\lambda_{0} + \alpha_{4}^{*}\lambda_{0}^{\nu(c)/2}+ \lambda_{0}^{-1} + \sum_{i=1}^{r}\lambda_{i} + \sum_{i=1}^{r}\eta_{i}^{*}\lambda_{i}^{-1}\right)\\
  &\quad\quad\quad+ L(\alpha_{1}^{*},\alpha_{2}^{*},\alpha_{3}^{*},\alpha_{4}^{*},\delta^{*},\eta^{*},r^{*})\\
  &= \rho^{*}v(\theta,\lambda) +
  L(\alpha_{1}^{*},\alpha_{2}^{*},\alpha_{3}^{*},\delta^{*},\eta^{*},r^{*})
  \;,
\end{align*}
where $\rho^{*} < 1$ and
$L(\alpha_{1}^{*},\alpha_{2}^{*},\alpha_{3}^{*},\delta^{*},\eta^{*},r^{*})
< \infty$.  Therefore, the chain is geometrically ergodic for
$r=r^{*}$, which implies that it is geometrically ergodic for all
$r\in (0,1)$. This proves Proposition 1.

\nocite{abra:2016, segura2011bounds,mccu:sear:neuh:2008,abramowitz1966handbook,ismail1978monotonicity,khare2011spectral}

\vspace*{5mm}

\noindent {\bf \large Acknowledgment}.  The second author was
supported by NSF Grant DMS-15-11945.

\pagebreak
\noindent {\LARGE \bf Appendices}
\begin{appendix}

\numberwithin{equation}{section}
\vspace*{-3mm}

\section{Derivation of $\pi(\theta|\tau,\lambda)$}

From (\ref{eq:fullposterior}),
\begin{align*}
\pi(\theta|\tau,\lambda,y) &\propto \exp\left\{-\frac{\lambda_{0}}{2}(y-W\theta)^{T}(y-W\theta)\right\}\exp\left\{-\frac{\lambda_{0}}{2}\beta^{T}D_{\tau}^{-1}\beta \right\}\exp\left\{-\frac{1}{2}u^{T}\Lambda u\right\} \\
&\propto  \exp\left\{-\frac{\lambda_{0}}{2}(y-W\theta)^{T}(y-W\theta)\right\}\exp\left\{-\frac{1}{2}\theta^{T}C\theta\right\},
\end{align*}
where
\begin{equation*}
C = \left[\begin{array}{cc}
\lambda_{0}D_{\tau}^{-1} & 0 \\
0 & \Lambda
\end{array}\right].
\end{equation*}
Thus,
\begin{align*}
  \pi(\theta|\tau,\lambda,y) &\propto \exp\left\{-\frac{1}{2}\left[\theta^{T}(\lambda_{0}W^{T}W + C)\theta -2\theta^{T}(\lambda_{0}W^{T}y)\right]\right\}\\
  &\propto \exp\left\{-\frac{1}{2}\left[\theta^{T}(\lambda_{0}W^{T}W +
      C)\theta -2\theta^{T}(\lambda_{0}W^{T}W + C)(\lambda_{0}W^{T}W +
      C)^{-1}(\lambda_{0}W^{T}y)\right]\right\} \;.
\end{align*}
Therefore, conditional on $\tau$, $\lambda$ and $y$, $\theta$ is
multivariate normal with mean $(\lambda_{0}W^{T}W +
C)^{-1}(\lambda_{0}W^{T}y)$ and covariance matrix $(\lambda_{0}W^{T}W
+ C)^{-1}$.  It is now left for us to compute these two values. From
the definition of $W$ and $C$,
\begin{equation*}
(\lambda_{0}W^{T}W + C)^{-1} = \left[\begin{array}{cc}
\lambda_{0}X^{T}X + \lambda_{0}D_{\tau}^{-1} & \lambda_{0}X^{T}Z\\
\lambda_{0}Z^{T}X & \lambda_{0}Z^{T}Z + \Lambda
\end{array}\right]^{-1} \;.
\end{equation*}
We will make use of the following inverse formula for block matrices
\begin{equation}
\label{eq:inverse_eq}
\left[\begin{array}{cc}
A & B \\
C & D
\end{array}\right]^{-1} = \left[\begin{array}{cc}
  A^{-1} + A^{-1}B(D-CA^{-1}B)^{-1}CA^{-1} & -A^{-1}B(D-CA^{-1}B)^{-1} \\
  -(D-CA^{-1}B)^{-1}CA^{-1} & (D - CA^{-1}B)^{-1}
\end{array}\right] \;.
\end{equation}
Let $\tsl = \lambda_{0}(X^{T}X + D_{\tau}^{-1})$, $\msl = I
-\lambda_{0}XT^{-1}_{\lambda,\tau}X^{T}$, and $\qsl=
\lambda_{0}Z^{T}M_{\lambda,\tau}Z + \Lambda$. Then,
\begin{equation*}
(\lambda_{0}W^{T}W + C)^{-1} = \left[\begin{array}{cc}
\tsl & \lambda_{0}X^{T}Z \\
\lambda_{0}Z^{T}X & \lambda_{0}Z^{T}Z + \Lambda
\end{array}\right]^{-1} := \left[\begin{array}{cc}
\Omega_{11} & \Omega_{12} \\
\Omega_{21} & \Omega_{22}
\end{array}\right] := \Omega \;.
\end{equation*}
From (\ref{eq:inverse_eq}),
\begin{align*}
  \Omega_{22} &=\left[\lambda_{0}Z^{T}Z + \Lambda - (\lambda_{0}Z^{T}X)\tsl^{-1}(\lambda_{0}X^{T}Z)\right]^{-1} \\
  & = \left[\lambda_{0} Z^{T}(I-\lambda_{0}X\tsli X^{T})Z + \Lambda \right] \\
  & = \left[\lambda_{0}Z^{T}\msl Z + \Lambda\right]^{-1} \\
  & = \qsli \;.
\end{align*}
Next,
\begin{align*}
  \Omega_{11} &= \tsl^{-1} + \tsl^{-1}(\lambda_{0}X^{T}Z)[\lambda_{0}Z^{T}Z + \Lambda - (\lambda_{0}Z^{T}X)\tsl^{-1}(\lambda_{0}X^{T}Z)]^{-1}(\lambda_{0}Z^{T}X)\tsl^{-1},\\
  &= \tsl^{-1} + \lambda_{0}^{2}\tsl^{-1}X^{T}Z\qsli Z^{T}X\tsl^{-1} \;. \\
\end{align*}
Similarly,
\begin{equation*}
  \Omega_{12} = -\tsli(\lambda_{0}X^{T}Z)[\lambda_{0}Z^{T}Z + \Lambda - (\lambda_{0}Z^{T}X)\tsl^{-1}(\lambda_{0}X^{T}Z)]^{-1} = -\lambda_{0}\tsl^{-1}X^{T}Z\qsl^{-1} \;,
\end{equation*}
and
\begin{equation*}
  \Omega_{21} = -\lambda_{0}\qsl^{-1}Z^{T}X\tsl^{-1} \;.
\end{equation*}
Hence,
\begin{equation}
  \Var[\theta|\tau,\lambda,y] = \left[\begin{array}{cc}
      \tsl^{-1} + \lambda_{0}^{2}\tsl^{-1}X^{T}Z\qsl^{-1}Z^{T}X\tsl^{-1} & -\lambda_{0}\tsl^{-1}X^{T}Z\qsl^{-1}\\
      -\lambda_{0}\qsl^{-1}Z^{T}X\tsl^{-1} & \qsl^{-1}
    \end{array}\right] .
\end{equation}

Now we just need to compute the conditional mean of $\theta$ given $\tau$, $\lambda$ and $y$.  Notice that
\begin{equation*}
\Ex[\theta|\tau,\lambda,y] = \Omega(\lambda_{0}W^{T}y) = \left[\begin{array}{cc}
\Omega_{11} & \Omega_{12}\\
\Omega_{21} & \Omega_{22}
\end{array}\right]\,\left[\begin{array}{c}
  \lambda_{0}X^{T}y \\
  \lambda_{0}Z^{T}y
\end{array}\right],
\end{equation*}
and
\begin{align*}
  \Omega_{11}(\lambda_{0}X^{T}y) + \Omega_{12}(\lambda_{0}Z^{T}y) &= \lambda_{0}(\tsli +\lambda_{0}^{2}\tsli X^{T}Z\qsli Z^{T}X\tsli)X^{T}y -\lambda_{0}(\lambda_{0}\tsli X^{T}Z\qsli )Z^{T}y\\
  &= \lambda_{0}\tsli X^{T}y + \lambda_{0}^{3}\tsli X^{T}Z\qsli Z^{T}X\tsli X^{T}y - \lambda_{0}^{2}\tsli X^{T}Z\qsli Z^{T}y \\
  &= \lambda_{0}\tsli X^{T}y - \lambda_{0}^{2}\tsli X^{T}Z\qsli Z^{T}(I - \lambda_{0}X\tsli X^{T})y \\
  &= \lambda_{0}\tsli X^{T}y - \lambda_{0}^{2}\tsli X^{T}Z\qsli
  Z^{T}\msl y \;.
\end{align*}
Additionally, we have
\begin{align*}
  \Omega_{21}(\lambda_{0}X^{T}y) + \Omega_{22}(\lambda_{0}Z^{T}y) &= \lambda_{0}(-\lambda_{0}\qsli Z^{T}X\tsli)X^{T}y +\lambda_{0}\qsli Z^{T}y\\
  &= \lambda_{0}\qsli Z^{T}(I - \lambda_{0}X\tsli X^{T})y\\
  &= \lambda_{0}\qsli Z^{T}\msl y \;.
\end{align*}
Hence
\begin{equation}
  \Ex[\theta|\tau,\lambda,y] = \left[\begin{array}{c}
      \lambda_{0} \tsli X^{T}y - \lambda_{0}^{2}\tsli X^{T}Z\qsli Z^{T}\msl y\\
      \lambda_{0}\qsli Z^{T}\msl y
\end{array}\right].
\end{equation}

Thus, the full conditional distribution of $\theta$ is multivariate normal with mean and covariance matrix given by (\ref{eq:theta_cond_mean}) and (\ref{eq:theta_cond_var}), respectively.

\section{Proof that the Markov chain
  $\{(\theta_{k},\lambda_{k})\}_{k=0}^{\infty}$ is Feller}
To prove that the Markov chain generated by the hybrid sampler is a
Feller chain, we must show that for each open set $O$, $P(\cdot,O)$ is
a lower semi-continuous function on $\X$.  Let
$(\theta_{l},\lambda_{l})$ be a sequence in $\X$ converging to
$(\theta,\lambda)\in \X$. Then,
\begin{align*}
  \liminf_{l\rightarrow\infty} P((\theta_{l},\lambda_{l}),O) & \ge \liminf_{l \rightarrow\infty} \,r\int_{\mathbb{R}^{p+q}}I_{O}(\tilde{\theta},\lambda_{l})\left[\int_{\mathbb{R}^{p}_{+}}\pi(\tilde{\theta}|\tau,\lambda_{l})\pi(\tau|\theta_{l},\lambda_{l})d\tau\right]\, d\tilde{\theta} \\
  & \quad \quad + \liminf_{l\rightarrow\infty} \,(1-r) \int_{\mathbb{R}^{m}_{+}} I_{O}(\theta_{l},\tilde{\lambda})\left[\int_{\mathbb{R}^{p}_{+}}\pi(\tilde{\lambda}|\theta_{l},\tau,)\pi(\tau|\theta_{l},\lambda_{l})\, d\tau\right]\,d\tilde{\lambda} \\
  & \geq  r\int_{\mathbb{R}^{p+q}}I_{O}(\tilde{\theta},\lambda)\left[\int_{\mathbb{R}^{p}_{+}}\liminf_{l\rightarrow\infty}\, \pi(\tilde{\theta}|\tau,\lambda_{l})\pi(\tau|\theta_{l},\lambda_{l})d\tau\right]\, d\tilde{\theta}\\
  & \quad \quad + (1-r)\int_{\mathbb{R}^{m}_{+}}I_{O}(\theta,\tilde{\lambda})\left[\int_{\mathbb{R}^{p}_{+}}\liminf_{l\rightarrow\infty}\,\pi(\tilde{\lambda}|\theta_{l},\tau)\pi(\tau|\theta_{l},\lambda_{l})\, d\tau\right]\,d\tilde{\lambda} \\
  & =  r\int_{\mathbb{R}^{p+q}}I_{O}(\tilde{\theta},\lambda)\left[\int_{\mathbb{R}^{p}_{+}} \pi(\tilde{\theta}|\tau,\lambda) \pi(\tau|\theta,\lambda) \, d\tau\right]\, d\tilde{\theta} \\
  & \quad \quad +
  (1-r)\int_{\mathbb{R}^{m}_{+}}I_{O}(\theta,\tilde{\lambda})\left[\int_{\mathbb{R}^{p}_{+}}
    \pi(\tilde{\lambda}|\theta,\tau)\pi(\tau|\theta,\lambda) \,
    d\tau\right]\,d\tilde{\lambda} \\ & = P((\theta,\lambda),O) \;,
\end{align*}
where the penultimate equality follows from the fact that all three
conditional densities are continuous in the conditioning variables
\citep{abra:2016}.

\section{Proof that $v(\theta,\lambda)$ is unbounded off compact
  sets}

Recall that the drift function $v(\theta,\lambda)$ is given by
\begin{equation*}
\begin{aligned}
  v(\theta,\lambda) & = \alpha_{1} \norm{y-W\theta}^{2} + \alpha_{2}
  \norm{\beta}^{2}+ \sum_{j=1}^{p} \frac{1}{\abs{\beta_{j}}^{\nu(c)}}
  + \sum_{i=1}^{m} \delta_{i} \norm{u_{i}}^{2} \\
  & \hspace*{35mm} + \alpha_{3} \lambda_{0} + \alpha_{4}
  \lambda_{0}^{\nu(c)/2} + \lambda_{0}^{-1} + \sum_{i=1}^{m}
  \lambda_{i} + \sum_{i=1}^{m}\eta_{i}\lambda_{i}^{-1} \;.
\end{aligned}
\end{equation*}
We need to show that this function is unbounded off compact sets; that
is, we must demonstrate that, for every $d \in \mathbb{R}$, the set
\[
S_{d} := \{(\theta,\lambda)\in \mathsf{X}: v(\theta,\lambda) \leq d\}
\;,
\]
is compact.  Let $d$ be such that $S_d$ is nonempty (otherwise $S_d$
is trivially compact).  Since $v(\theta,\lambda)$ is continuous on
$\X$, $S_{d}$ is a closed set.  Now define
\begin{align*}
  A_{j} &= \left\{\beta_{j}\in\mathbb{R} \setminus \{0\}:\alpha_{2}\beta_{j}^{2}+\frac{1}{\abs{\beta_{j}}^{\nu(c)}}\leq d\right\}, \quad j=1,\ldots,p \;,\\
  B_{i} &= \{u_{i}\in\mathbb{R}: \delta_{i}u_{i}^{2} \leq d\} \;, \quad i=1,\ldots,m \;,\\
  C_{0} &= \{\lambda_{0}\in \mathbb{R}_{+}: \alpha_{3}\lambda_{0} + \lambda_{0}^{-1}+\alpha_{4}\lambda_{0}^{\nu(c)/2}\leq d\} \;, \\
  C_{i} &=
  \{\lambda_{i}\in\mathbb{R}_{+}:\lambda_{i}+\eta_{i}\lambda_{i}^{-1}\leq
  d\}, \quad i=1,\ldots,m \;.
\end{align*}
All of the above sets are closed and bounded, and thus the set
\begin{equation*}
  T_{d} := \bigtimes_{j=1}^{p}A_{i}\times\bigtimes_{j=1}^{m}B_{j}\times\bigtimes_{i=0}^{m}C_{i} \;,
\end{equation*}
is a compact set in $\X$.  Since, $S_{d}$ is a closed set and $S_{d}\subseteq T_{d}$, it follows that $S_{d}$ is a compact set in $\X$.

\section{Proofs of the lemmas}
\label{app:lemmas}

\subsection{Preliminary results}
\label{app:prelim}

We begin by introducing some notation and stating a few necessary
facts about non-negative definite matrices. Note that if $C$ is a
non-negative definite matrix then $\tr(C)$ is non-negative. If $A,B
\in \mathbb{R}^{n\times n}$ are symmetric matrices such that $B-A$ is
non-negative definite, we write $A \preceq B$.  Similarly, if $B-A$ is
positive definite, we write $A \prec B$. Additionally, if $A\preceq
B$, then $\tr(A) \leq \tr(B)$. Furthermore, if $A$ and $B$ are
positive definite matrices then $A \preceq B \Leftrightarrow B^{-1}
\preceq A^{-1}$. Lastly, for a matrix $D$ we let $\norm{D}$ represent
the Frobenius norm of the matrix $\norm{D} := \sqrt{\tr(D^{T}D)}$.

We will also require the singular value decomposition of several
matrices in our proofs, thus it will be helpful to establish some
common notation. For a matrix $A \in \mathbb{R}^{n\times m}$, let
$k_{A} = \rank(A) \leq \min\{n,m\}$ and denote the singular value
decomposition of $A$ by $U_{A} \Gamma_{A} V_{A}^{T}$, where $U_{A}$
and $V_{A}$ are orthogonal matrices of dimension $n$ and $m$,
respectively, and
\begin{equation*}
\Gamma_{A} := \left[\begin{array}{cc}
           \Gamma_{A}^{*} & 0_{k_{A},\,m-k_{A}}\\
           0_{n-k_{A},\,k_{A}} & 0_{n-k_{A},\,m-k_{A}}
           \end{array}
     \right] \;,
\end{equation*}
where $\Gamma_{A}^{*} := \text{diag}\{\gamma_{A\,1},\ldots,\gamma_{A\,k}\}$. The values
$\gamma_{A\,1},\ldots,\gamma_{A\,k}$ are the singular values of $A$, which are
strictly positive. We denote $\gamma_{A\max}$ as the largest singular value of $A$. Lastly, in an abuse of notation, $\gamma_{A\, i}^{2} := 0$ whenever $i> k_{A}$.

In order to prove Lemmas \ref{lemma:theta_tr_ineq} -
\ref{lemma:u_ineq}, we will need the following results.
\begin{lemma}
\label{lemma:m_mat_lem}
For all $\tau \in \mathbb{R}^{p}_{+}$ and $\lambda \in
\mathbb{R}^{p}_{+}$ \;,
\begin{itemize}
\item[(1)] $\msl = UR_{\lambda, \tau}U^{T}$ where $U\in
  \mathbb{R}^{n\times n}$ is an orthogonal matrix and
  $R_{\lambda,\tau}:= \mbox{diag}\{r_{1},r_{2},\ldots,r_{n}\}$ where
  $0 < r_{i} \leq 1$ for $i=1,2,\ldots,n$.
    \item[(2)]$0 \prec (\tau_{\max}s_{\max}^{2}+1)^{-1}I \preceq \msl \preceq I$, where $\tau_{\max}:=\max\{\tau_{1},\ldots,\tau_{p}\}$ and $s_{\max}$ is the largest singular value of the matrix $X$.
    \item[(3)] $||\msl|| \leq \sqrt{n}$.
\end{itemize}
\end{lemma}

\begin{proof}[Proof of Lemma~\ref{lemma:m_mat_lem}]
  This proof is similar to the proof of Lemmas 4 and 5 of
  \cite{roma:hobe:2015}. Recall,
\begin{align*}
\msl &= I - \lambda_{0}X\tsli X^{T} \\
&= I -\lambda_{0}X[\lambda_{0}(X^{T}X + D_{\tau}^{-1})]^{-1}X^{T}\\
&= I - X(X^{T}X+D_{\tau}^{-1})^{-1}X^{T}\\
&= I - XD_{\tau}^{1/2}(D_{\tau}^{1/2}X^{T}XD_{\tau}^{1/2} +I_{p})^{-1}D_{\tau}^{1/2}X^{T}.
\end{align*}
Let $B:= XD_{\tau}^{1/2}$ and let $U_{B}\Gamma_{B}V_{B}^{T}$ be the singular value decomposition of $B$.
Then
\begin{align*}
\msl &= I - U_{B}\Gamma_{B} V_{B}^{T}(V_{B}\Gamma_{B}^{T}U_{B}^{T}U_{B}\Gamma_{B} V_{B}^{T} + I_{p})^{-1}V_{B}\Gamma_{B}^{T}U_{B}^{T}\\
&= I - U_{B}\Gamma_{B}(\Gamma_{B}^{T}\Gamma_{B} + I_{p})^{-1}\Gamma_{B}^{T}U_{B}^{T} \\
& = U_{B}(I-\Gamma_{B}(\Gamma_{B}^{T}\Gamma_{B} + I_{p})^{-1}\Gamma_{B}^{T})U_{B}^{T}\\
&= U_{B}R_{\lambda, \tau}U_{B}^{T}
\;,
\end{align*}
where $R_{\lambda, \tau} := \mbox{diag}\{r_{1},r_{2},\ldots,r_{n}\}$, with
\begin{equation*}
r_{i} = 1-\frac{\gamma^{2}_{B\,i}}{\gamma_{B\,i}^{2}+1} = \frac{1}{\gamma_{B\,i}^{2}+1}\;,
\end{equation*}
where $\gamma_{B\,i}^{2} = 0$ for $i > k_{B}$, and $0<r_{i}\leq 1$ for $i=1,\ldots,n$. This proves (1).

Next, let $\tau_{\max}:=\max\{\tau_{1},\tau_{2},\ldots,\tau_{p}\}$, and notice that
\begin{align*}
X(X^{T}X + D_{\tau}^{-1})^{-1}X^{T}&\preceq X(X^{T}X + \tau_{\max}^{-1}I_{p})^{-1}X^{T}\\
&= U_{X}\Gamma_{X}V_{X}^{T}(V_{X}\Gamma_{X}^{T}U_{X}^{T}U_{X}\Gamma_{X}V_{X}^{T}+\tau_{\max}^{-1}I_{p})^{-1}V_{X}\Gamma_{X}^{T}U_{X}^{T}\\
&= U_{X}\Gamma_{X}(\Gamma_{X}^{T}\Gamma_{X}+\tau_{\max}^{-1}I_{p})^{-1}\Gamma_{X}^{T}U_{X}^{T} \;.
\end{align*}
Thus,
\begin{align*}
\msl &= I - \lambda_{0}X\tsli X^{T}\\
&= I - X(X^{T}X + D_{\tau}^{-1})^{-1}X^{T}\\
&\succeq I - U_{X}\Gamma_{X}(\Gamma_{X}^{T}\Gamma_{X}+\tau_{\max}^{-1}I_{p})^{-1}\Gamma_{X}^{T}U_{X}^{T}\\
&= U_{X}[I-\Gamma_{X}(\Gamma_{X}^{T}\Gamma_{X}+\tau_{\max}^{-1}I_{p})^{-1}\Gamma_{X}^{T}]U_{X}^{T} \;.
\end{align*}
Notice that
$[I-\Gamma_{X}(\Gamma_{X}^{T}\Gamma_{X}+\tau_{\max}^{-1}I_{p})^{-1}\Gamma_{X}^{T}]$
is a diagonal matrix whose entries are given by
$$1-\frac{\gamma_{X\, i}^{2}}{\gamma_{X\, i}^{2}+\tau_{\max}^{-1}} = \frac{\tau_{\max}^{-1}}{\gamma_{X\, i}^{2}+\tau_{\max}^{-1}} = \frac{1}{\gamma_{X\, i}^{2}\,\tau_{\max}+1}\geq \frac{1}{s_{\max}^{2}\tau_{\max}+1} \;,$$
for all $i=1,\ldots,n$ where $s_{\max}$ is the largest singular value of $X$.

Thus,
\begin{equation*}
0 \prec \left(\frac{1}{\tau_{\max}s_{\max}^{2} + 1}\right)I = U_{X}\left(\frac{1}{\tau_{\max}s_{\max}^{2} + 1}\right)U_{X}^{T} \preceq \msl = U_{B}R_{\lambda,\tau}\,U_{B}^{T} \preceq U_{B}U_{B}^{T} = I,
\end{equation*}
which proves (2).

Lastly,
$$\norm{\msl}^{2} = \tr(\msl^{2}) = \tr(U_{B}R_{\lambda,\tau}U_{B}^{T}U_{B}R_{\lambda,\tau}U_{B}^{T}) = \tr(U_{B}R^{2}_{\lambda,\tau}U_{B}^{T}) \leq \tr(U_{B}U_{B}^{T}) = n \;,$$
and therefore $||\msl|| \leq \sqrt{n}$.
\end{proof}

\begin{lemma}
\label{lemma:ZQ_lemma}
For all $\lambda \in \mathbb{R}^{m+1}_{+}$ and all $\tau \in \mathbb{R}^{p}_{+}$,
\begin{itemize}
    \item[(1)] $\lambda_{0}\msl^{1/2}Z\qsli Z^{T}\msl^{1/2} \preceq I_{n}$ .
    \item[(2)] $\norm{\lambda_{0}\msl^{1/2}Z\qsli Z^{T}\msl^{1/2}}^{2} \leq n$ .
\end{itemize}

\end{lemma}

\begin{proof}[Proof of Lemma~\ref{lemma:ZQ_lemma}]
Notice that
\begin{equation}
\label{eq:ZQ_eq1}
\begin{aligned}
\lambda_{0}\msl^{1/2}Z\qsli Z^{T}\msl^{1/2} &= \lambda_{0}\msl^{1/2}Z(\lambda_{0}Z^{T}\msl Z +\Lambda)^{-1} Z^{T}\msl^{1/2}\\
&=\msl^{1/2}Z(Z^{T}\msl Z +\lambda_{0}^{-1}\Lambda)^{-1} Z^{T}\msl^{1/2}\\
&=\msl^{1/2}Z\Lambda^{-1/2}(\Lambda^{-1/2}Z^{T}\msl^{1/2}\msl^{1/2}Z\Lambda^{-1/2} +\lambda_{0}^{-1}I)^{-1}\Lambda^{-1/2} Z^{T}\msl^{1/2} \;,
\end{aligned}
\end{equation}
Let $A:=\msl^{1/2}Z\Lambda^{-1/2}$ and let $U_{A}\Gamma_{A}V_{A}^{T}$ represent the singular value decomposition of $A$. Then, from (\ref{eq:ZQ_eq1}),
\begin{equation}
\label{eq:ZQ_eq2}
\begin{aligned}
\lambda_{0}\msl^{1/2}Z\qsli Z^{T}\msl^{1/2} &= U_{A}\Gamma_{A}V_{A}^{T}(V_{A}\Gamma_{A}^{T}U_{A}^{T}U_{A}\Gamma_{A}V_{A}^{T}+\lambda_{0}^{-1}I)^{-1}V_{A}\Gamma_{A}^{T}U_{A}^{T}\\
&= U_{A}\Gamma_{A}(\Gamma_{A}^{T}\Gamma_{A}+\lambda_{0}^{-1}I)^{-1}\Gamma_{A}^{T}U_{A}^{T}\;,
\end{aligned}
\end{equation}
where $\Gamma_{A}(\Gamma_{A}^{T}\Gamma_{A}+\lambda_{0}^{-1}I)^{-1}\Gamma_{A}^{T}$ is a diagonal matrix whose elements are given by
\begin{equation}
\label{eq:ZQ_eq3}
\frac{\gamma_{A\,i}^{2}}{\gamma_{A\,i}^{2}+\lambda_{0}^{-1}} \leq 1 \ \ \text{for} \ i=1,\ldots,n \;.
\end{equation}
Thus,
\begin{equation}
\label{eq:ZQ_eq4}
\lambda_{0}\msl^{1/2}Z\qsli Z^{T}\msl^{1/2} = U_{A}\Gamma_{A}(\Gamma_{A}^{T}\Gamma_{A}+\lambda_{0}^{-1}I)^{-1}\Gamma_{A}^{T}U_{A}^{T} \preceq U_{A}U_{A}^{T} = I_{n} \;,
\end{equation}
which proves (1).

Next, from (\ref{eq:ZQ_eq2}) and (\ref{eq:ZQ_eq3})
\begin{equation}
\label{eq:y_norm_ineq6}
\begin{aligned}
\norm{\lambda_{0}\msl^{1/2}Z\qsli Z^{T}\msl^{1/2}}^{2} &= \tr[(\lambda_{0}\msl^{1/2}Z\qsli Z^{T}\msl^{1/2})^{T}(\lambda_{0}\msl^{1/2}Z\qsli Z^{T}\msl^{1/2})] \\
&= \tr[(U_{A}\Gamma_{A}(\Gamma_{A}^{T}\Gamma_{A}+\lambda_{0}^{-1}I)^{-1}\Gamma_{A}^{T}U_{A}^{T})^{T}U_{A}\Gamma_{A}(\Gamma_{A}^{T}\Gamma_{A}+\lambda_{0}^{-1}I)^{-1}\Gamma_{A}^{T}U_{A}^{T}] \\
&= \tr[U_{A}\Gamma_{A}(\Gamma_{A}^{T}\Gamma_{A}+\lambda_{0}^{-1}I)^{-1}\Gamma_{A}^{T}U_{A}^{T}U_{A}\Gamma_{A}(\Gamma_{A}^{T}\Gamma_{A}+\lambda_{0}^{-1}I)^{-1}\Gamma_{A}^{T}U_{A}^{T}] \\
&= \tr[U_{A}\Gamma_{A}(\Gamma_{A}^{T}\Gamma_{A}+\lambda_{0}^{-1}I)^{-1}\Gamma_{A}^{T}\Gamma_{A}(\Gamma_{A}^{T}\Gamma_{A}+\lambda_{0}^{-1}I)^{-1}\Gamma_{A}^{T}U_{A}^{T}]\\
&=\tr[\Gamma_{A}(\Gamma_{A}^{T}\Gamma_{A}+\lambda_{0}^{-1}I)^{-1}\Gamma_{A}^{T}\Gamma_{A}(\Gamma_{A}^{T}\Gamma_{A}+\lambda_{0}^{-1}I)^{-1}\Gamma_{A}^{T}U_{A}^{T}U_{A}]\\
&=\tr[\Gamma_{A}(\Gamma_{A}^{T}\Gamma_{A}+\lambda_{0}^{-1}I)^{-1}\Gamma_{A}^{T}\Gamma_{A}(\Gamma_{A}^{T}\Gamma_{A}+\lambda_{0}^{-1}I)^{-1}\Gamma_{A}^{T}]\\
&= \sum_{i=1}^{n}\left(\frac{\gamma_{A\,i}^{2}}{\gamma_{A\,i}^{2}+\lambda_{0}^{-1}}\right)^{2}\\
&\leq n\;,
\end{aligned}
\end{equation}
which proves (2).
\end{proof}

\begin{lemma}
\label{lemma:exp_beta_lem}
For all $\lambda\in \mathbb{R}^{m+1}_{+}$ and $\tau \in \mathbb{R}^{p}_{+}$
\begin{itemize}
    \item[(1)] $\norm{\lambda_{0}\tsli X^{T}}^{2} < \infty$ .
    \item[(2)] $\norm{I-\lambda_{0}Z\qsli Z^{T}\msl}^{2} \leq n^{2}\left(s_{\max}^{2}\sum_{j=1}^{p}\tau_{j}+1\right)$ .
\end{itemize}
\end{lemma}

\begin{proof}[Proof of Lemma~\ref{lemma:exp_beta_lem}]

\begin{equation}
\begin{aligned}
\norm{\lambda_{0}\tsli X^{T}}^{2} &= \norm{(X^{T}X+D_{\tau}^{-1})^{-1}X^{T}}^{2}\\
&= \tr(X(X^{T}X+D_{\tau}^{-1})^{-2}X^{T})\\
&= \sum_{i=1}^{n}e_{i}^{T}X(X^{T}X+D_{\tau}^{-1})^{-2}X^{T}e_{i}\\
&= \sum_{i=1}^{n}\norm{(X^{T}X+D_{\tau}^{-1})^{-1}X^{T}e_{i}}^{2}
\;,
\end{aligned}
\end{equation}
where $e_{i}\in\mathbb{R}^{n}$, $i=1,\ldots,n$, are the standard unit vectors.  Let $x_{i}$ represent the $i$th column vector of $X^{T}$.

For a given $i$,
\begin{equation}
\begin{aligned}
\norm{(X^{T}X+D_{\tau}^{-1})^{-1}X^{T}e_{i}}^{2}&= \norm{(X^{T}X+D_{\tau}^{-1})^{-1}x_{i}}^{2}\\
&=\norm{(\sum_{j=1}^{n}x_{j}x_{j}^{T}+D_{\tau}^{-1})^{-1}x_{i}}^{2}\\
&=\norm{(x_{i}x_{i}^{T}+\sum_{i\neq j}^{n}x_{j}x_{j}^{T}+D_{\tau}^{-1})^{-1}x_{i}}^{2}\\
&= \norm{(x_{i}x_{i}^{T}+\sum_{i\neq j}^{n}x_{j}x_{j}^{T}+(D_{\tau}^{-1}-\tau_{\bullet}^{-1}I_{p}) + \tau_{\bullet}^{-1}I_{p})^{-1}x_{i}}^{2}\;,\\
\end{aligned}
\end{equation}
where $\tau_{\bullet}^{-1} = (\tau_{1}+\ldots+\tau_{p})^{-1}$. Define the vectors $t_{1},t_{2},\ldots,t_{n+p} \in \mathbb{R}^{p}$ such that $t_{j} = x_{j}$ for $j = 1,\ldots,n$, and $t_{n+k}=e_{k}$, $k=1,\ldots,p$, where $e_{k}$ are the standard unit vectors in $\mathbb{R}^{p}$.  Next, define the positive constants $w_{1},w_{2},\ldots,w_{n+p}\in\mathbb{R}_{+}$ as follows
\begin{equation*}
w_{l} = \left\{\begin{array}{ll}
\tau_{\bullet}^{-1} & l=i \;,\\
1 & l\neq i, 1\leq l\leq p \;,\\
\tau_{1}^{-1}-\tau_{\bullet}^{-1} & l=n+1\;, \\
\vdots & \vdots\\
\tau_{p}^{-1}-\tau_{\bullet}^{-1} & l=n+p \;.
\end{array}\right.
\end{equation*}
Thus,
\begin{equation}
\begin{aligned}
\norm{(X^{T}X+D_{\tau}^{-1})^{-1}X^{T}e_{i}}^{2}&=\norm{(t_{i}t_{i}^{t} + \sum_{l\in\{1,2,\ldots,n\}\diagdown \{i\}}w_{l}t_{l}t_{l}^{T} + \sum_{l=n+1}^{n+p}w_{l}t_{l}t_{l}^{T} + w_{i}I_{p})^{-1}t_{i}}^{2}\\
&= t_{i}^{T}\left(t_{i}t_{i}^{t} + \sum_{l\in\{1,2,\ldots,n\}\diagdown \{i\}}w_{l}t_{l}t_{l}^{T} + \sum_{l=n+1}^{n+p}w_{l}t_{l}t_{l}^{T} + w_{i}I_{p}\right)^{-2}t_{i}\\
&\leq \sup_{w\in\mathbb{R}^{n+p}} t_{i}^{T}\left(t_{i}t_{i}^{t} + \sum_{l\in\{1,2,\ldots,n\}\diagdown \{i\}}w_{l}t_{l}t_{l}^{T} + \sum_{l=n+1}^{n+p}w_{l}t_{l}t_{l}^{T} + w_{i}I_{p}\right)^{-2}t_{i} \\
&:= C_{i}^{*}< \infty\;, \quad i=1 ,\ldots,n \;,
\end{aligned}
\end{equation}
where the last inequality follows from \cite{khare2011spectral}. Thus
\begin{equation}
\norm{\lambda_{0}\tsli X^{T}}^{2}  \leq \sum_{i=1}^{n}C_{i}^{*} < \infty\;,
\end{equation}
which proves (1).

Next,
\begin{equation}
\begin{aligned}
\norm{I-\lambda_{0}Z\qsli Z^{T}\msl}^{2} &= \tr[(I-\lambda_{0}Z\qsli Z^{T}\msl)^{T}(I-\lambda_{0}Z\qsli Z^{T}\msl)] \\
&= \tr[I - \lambda_{0}Z\qsli Z^{T}\msl - \lambda_{0}\msl Z\qsli Z^{T} + \lambda_{0}^{2}\msl Z\qsli Z^{T}Z\qsli Z^{T}\msl]\\
&= \tr(I) - 2\lambda_{0}\tr(Z\qsli Z^{T}\msl) + \lambda_{0}^{2}\tr(\msl Z\qsli Z^{T}Z\qsli Z^{T}\msl)\\
&= n - 2\lambda_{0}\tr(\msl^{1/2}Z\qsli Z^{T}\msl^{1/2}) + \lambda_{0}^{2}\tr(\msl Z\qsli Z^{T}Z\qsli Z^{T}\msl)\\
&\leq n + \lambda_{0}^{2}\tr(\msl Z\qsli Z^{T}Z\qsli Z^{T}\msl)\;.
\end{aligned}
\end{equation}
From (2) of Lemma~\ref{lemma:m_mat_lem},
\begin{equation}
\begin{aligned}
\msl Z\qsli Z^{T}Z\qsli Z^{T}\msl &= \msl Z\qsli Z^{T}\msl^{1/2}\msl^{-1}\msl^{1/2}Z\qsli Z^{T}\msl\\
&= \msl^{1/2}(\msl^{1/2} Z\qsli Z^{T}\msl^{1/2})\msl^{-1}(\msl^{1/2}Z\qsli Z^{T}\msl^{1/2})\msl^{1/2}\\
&\preceq (\tau_{\max}s_{\max}^{2}+1)\msl^{1/2}(\msl^{1/2} Z\qsli Z^{T}\msl^{1/2})^{2}\msl^{1/2}\;.
\end{aligned}
\end{equation}
From (2) of Lemma~\ref{lemma:ZQ_lemma} and (2) of Lemma~\ref{lemma:m_mat_lem},
\begin{equation}
\begin{aligned}
\lambda_{0}^{2}\tr(\msl Z\qsli Z^{T}Z\qsli Z^{T}\msl) &\leq (\tau_{\max}s_{\max}^{2}+1)\tr[\lambda_{0}^{2}\msl^{1/2}(\msl^{1/2} Z\qsli Z^{T}\msl^{1/2})^{2}\msl^{1/2}] \\
&= (\tau_{\max}s_{\max}^{2}+1)\norm{\lambda_{0}(\msl^{1/2} Z\qsli Z^{T}\msl^{1/2})\msl^{1/2}}^{2}\\
&\leq (\tau_{\max}s_{\max}^{2}+1)\norm{\lambda_{0}\msl^{1/2} Z\qsli Z^{T}\msl^{1/2}}^{2}\norm{\msl^{1/2}}^{2}\\
&\leq (\tau_{\max}s_{\max}^{2}+1)n\,\norm{\msl^{1/2}}^{2}\\
&= (\tau_{\max}s_{\max}^{2}+1)n\,\tr(\msl)\\
&\leq (\tau_{\max}s_{\max}^{2}+1)n\,\tr(I)\\
&=n^{2}(s_{\max}^{2}\tau_{\max}+1)\\
&\leq n^{2}\left(s_{\max}^{2}\sum_{j=1}^{p}\tau_{j}+1\right)
\;,
\end{aligned}
\end{equation}
which proves (2).
\end{proof}

\subsection{Proof of Lemma~\ref{lemma:theta_tr_ineq}}

\noindent {\bf Lemma ~\ref{lemma:theta_tr_ineq}.\;} \textit{For all
  $\tau \in \mathbb{R}^{p}_{+}$ and $\lambda \in \mathbb{R}^{m+1}_{+}
  \;,$}
\begin{itemize}
\item[(1)] $\tr(W\Var[\tilde{\theta}|\tau,\lambda]W^{T}) \leq
  \tr(X\tsli X^{T}) + \tr(Z\qsli Z^{T})$ ,

\item[(1)] $\tr(X\tsli X^{T}) \leq \rank(X)\lambda_{0}^{-1}$ ,
  \hspace*{1mm} \mbox{and}

\item[(2)] $\tr(Z\qsli Z^{T}) \leq
  \tr(ZZ^{T})\sum_{i=1}^{m}\lambda_{i}^{-1}$ .
\end{itemize}

\begin{proof}[Proof of Lemma~\ref{lemma:theta_tr_ineq}]

First,
\begin{align*}
 W\Var[\tilde{\theta}|\tau,\lambda]W^{T} &= [X \ Z]\left[\begin{array}{cc}
\tsl^{-1} + \lambda_{0}^{2}\tsl^{-1}X^{T}Z\qsl^{-1}Z^{T}X\tsl^{-1} & -\lambda_{0}\tsl^{-1}X^{T}Z\qsl^{-1}\\
-\lambda_{0}\qsl^{-1}Z^{T}X\tsl^{-1} & \qsl^{-1}
\end{array}\right]\left[\begin{array}{c}
X^{T}\\
Z^{T}
\end{array}\right]\\
&= X\tsli X^{T} +\lambda_{0}^{2}X\tsli X^{T}Z\qsli Z^{T}X\tsli X^{T} -\lambda_{0}X\tsli X^{T}Z\qsli Z^{T} \\
&\quad \quad - \lambda_{0}Z\qsli Z^{T}X\tsli X^{T} + Z\qsli Z^{T}.
\end{align*}
Notice that $I-\msl = \lambda_{0}X\tsli X^{T}$, and therefore
\begin{align*}
W\Var[\tilde{\theta}|\tau,\lambda]W^{T} &=  X\tsli X^{T} + (I-\msl)Z\qsli Z^{T}(I-\msl) - (I-\msl)Z\qsli Z^{T}\\
& \quad \quad  -Z\qsli Z^{T}(I-\msl) + Z\qsli Z^{T}\\
&=X\tsli X^{T} + (I-\msl)Z\qsli Z^{T}(I-\msl -I) \\
& \quad \quad-Z\qsli Z^{T}(I-\msl) + Z\qsli Z^{T}\\
&=X\tsli X^{T} - (I-\msl)Z\qsli Z^{T}(I+\msl) +  (I-\msl)Z\qsli Z^{T}\\
& \quad \quad-Z\qsli Z^{T}(I-\msl) + Z\qsli Z^{T}.
\end{align*}
Thus,
\begin{align*}
\tr(W\Var[\tilde{\theta}|\tau,\lambda]W^{T}) &= \tr(X\tsli X^{T}) - \tr((I-\msl)Z\qsli Z^{T}(I+\msl)) \\
& \quad \quad + \tr((I-\msl)Z\qsli Z^{T}) - \tr(Z\qsli Z^{T}(I-\msl)) + \tr(Z\qsli Z^{T}) \\
&= \tr(X\tsli X^{T}) - \tr((I-\msl)Z\qsli Z^{T}(I+\msl)) + \tr(Z\qsli Z^{T})\\
&= \tr(X\tsli X^{T}) - \tr(\qsl^{-1/2} Z^{T}(I-\msl^{2})Z\qsl^{-1/2}) + \tr(Z\qsli Z^{T}) \;.\\
\end{align*}
Applying (1) of Lemma \ref{lemma:m_mat_lem} and using the fact that $R_{\lambda,\tau}^{2} \preceq I$, we get
\begin{equation*}
\qsl^{-1/2} Z^{T}(I-\msl^{2})Z\qsl^{-1/2}= \qsl^{-1/2} Z^{T}U(I-R_{\lambda, \tau}^{2})U^{T}Z\qsl^{-1/2} \succeq 0 \;.
\end{equation*}
Hence $\tr(\qsl^{-1/2} Z^{T}(I-\msl^{2})Z\qsl^{-1/2}) \geq 0$, and therefore
\begin{equation}
\label{eq:theta_var_ineq}
\tr(W\Var[\tilde{\theta}|\tau,\lambda]W^{T}) \leq \tr(X\tsli X^{T})  + \tr(Z\qsli Z^{T}) \;,
\end{equation}
which proves (1),

Next, notice that
\begin{equation*}
X\tsli X^{T} = X[\lambda_{0}(X^{T}X + D_{\tau}^{-1})]^{-1}X^{T} = \lambda_{0}^{-1}X[X^{T}X+D_{\tau}^{-1}]^{-1}X^{T} \preceq \lambda_{0}^{-1}X[X^{T}X+\tau_{\max}^{-1}I_{p}]^{-1}X^{T}.
\end{equation*}
Then,
\begin{align*}
\tr(X\tsli X^{T}) &\leq \lambda_{0}^{-1}\tr[U_{X}\Gamma_{X}V^{T}_{X}(V_{X}\Gamma_{X}^{T}U^{T}_{X}U_{X}\Gamma_{X}V^{T}_{X}+\tau_{\max}^{-1}I_{p})^{-1}V_{X}\Gamma_{X}^{T}U^{T}_{X}]\\
& = \lambda_{0}^{-1}\tr[U_{X}\Gamma_{X}(\Gamma_{X}^{T}\Gamma_{X}+\tau_{\max}^{-1}I_{p})^{-1}\Gamma_{X}^{T}U_{X}^{T}]\\
& = \lambda_{0}^{-1}\tr[\Gamma_{X}(\Gamma_{X}^{T}\Gamma_{X}+\tau_{\max}^{-1}I_{p})^{-1}\Gamma_{X}^{T}U_{X}^{T}U_{X}]\\
& = \lambda_{0}^{-1}\tr[\Gamma_{X}(\Gamma_{X}^{T}\Gamma_{X}+\tau_{\max}^{-1}I_{p})^{-1}\Gamma_{X}^{T}]\\
&= \lambda_{0}^{-1}\sum_{i=1}^{k_{X}}\frac{\gamma_{i}^{2}}{\gamma_{i}^{2}+\tau_{\max}^{-1}} \\
&\leq \lambda_{0}^{-1}k_{X} = \lambda_{0}^{-1}\mbox{rank}(X) \;,
\end{align*}
which proves (2).

Finally,
\begin{equation*}
Z\qsli Z^{T} = Z(\lambda_{0}Z^{T}\msl Z+\Lambda)^{-1}Z^{T} \preceq Z\Lambda^{-1}Z^{T} \preceq \lambda_{\min}^{-1}ZZ^{T},
\end{equation*}
where $\lambda_{\min}:=\min\{\lambda_{1},\lambda_{2},\ldots,\lambda_{r}\}$. Thus,
\begin{equation*}
\tr(Z\qsli Z^{T}) \leq \tr(\lambda_{\min}^{-1}ZZ^{T}) \leq \lambda_{\min}^{-1}\tr(ZZ^{T}) \leq\tr(ZZ^{T})\sum_{i=1}^{m}\lambda_{i}^{-1},
\end{equation*}
which proves (3).
\end{proof}

\subsection{Proof of Lemma~\ref{lemma:ynorm_ineq}}

\noindent
{\bf Lemma~\ref{lemma:ynorm_ineq}.\;}\textit{For all $\tau \in \mathbb{R}^{p}_{+}$ and $\lambda \in \mathbb{R}^{m+1}_{+}$}
$$\norm{y-W\Ex[\theta|\tau,\lambda]}^{2} \leq  2n\norm{y}^{2}+ 2n^{3}\norm{y}^{2} \;.$$
\begin{proof}[Proof of Lemma~\ref{lemma:ynorm_ineq}]
First,
\begin{align*}
W\Ex[\theta|\tau,\lambda] &= [X \ Z]\left[\begin{array}{c}
\lambda_{0}\tsli X^{T}y - \lambda_{0}^{2}\tsli X^{T}Z\qsli Z^{T}\msl y\\
\lambda_{0}\qsli Z^{T}\msl y
\end{array}\right],\\
&= \lambda_{0}X\tsli X^{T}y-\lambda_{0}^{2}X\tsli X^{T}Z\qsli Z^{T}\msl y + \lambda_{0}Z\qsli Z^{T}\msl y \;.
\end{align*}
Thus,
\begin{equation}
\label{eq:y_norm_ineq1}
\begin{aligned}
\norm{y-W\Ex[\theta|\tau,\lambda]}^{2} &= \norm{y - \lambda_{0}X\tsli X^{T}y+\lambda_{0}^{2}X\tsli X^{T}Z\qsli Z^{T}\msl y - \lambda_{0}Z\qsli Z^{T}\msl y}^{2}\\
&=\norm{(I-\lambda_{0}X\tsli X^{T})y -(I-\lambda_{0}X\tsli X^{T})\lambda_{0}Z\qsli Z^{T}\msl y}^{2}\\
& = \norm{\msl y - \lambda_{0}\msl Z\qsli Z^{T}\msl y}^{2}\\
&\leq 2\norm{\msl y}^{2} + \norm{\lambda_{0}\msl Z\qsli Z^{T}\msl y}^{2} \\
&\leq 2\norm{\msl}^{2}\norm{y}^{2} + 2\norm{\lambda_{0}\msl Z\qsli Z^{T}\msl y}^{2}\\
&\leq 2n\norm{y}^{2} + 2\norm{\lambda_{0}\msl Z\qsli Z^{T}\msl y}^{2} \;,
\end{aligned}
\end{equation}
where the last inequality follows from (3) of Lemma \ref{lemma:m_mat_lem}.

Indeed,
\begin{equation}
\label{eq:y_norm_ineq2}
\begin{aligned}
\norm{\lambda_{0}\msl Z\qsli Z^{T}\msl y}^{2} & = \norm{\lambda_{0}\msl^{1/2}\msl^{1/2} Z\qsli Z^{T}\msl^{1/2}\msl^{1/2} y}^{2},\\
&\leq \norm{\msl^{1/2}}^{2}\norm{\lambda_{0}\msl^{1/2}Z\qsli Z^{T}\msl^{1/2}}^{2}\norm{\msl^{1/2}y}^{2}\\
&\leq \norm{\msl^{1/2}}^{2}\norm{\lambda_{0}\msl^{1/2}Z\qsli Z^{T}\msl^{1/2}}^{2}\norm{\msl^{1/2}}^{2}\norm{y}^{2}\\
&\leq n^{3}\norm{y}^{2} \;,
\end{aligned}
\end{equation}
where the last inequality follows from (2) of Lemma~\ref{lemma:ZQ_lemma}, and (2) of Lemma \ref{lemma:m_mat_lem} since
\begin{align*}
\norm{\msl^{1/2}}^{2} = \tr(\msl) \leq \tr(I_{n}) =  n \;.
\end{align*}
\end{proof}
Thus,
$$\norm{y-W\Ex[\theta|\tau,\lambda]}^{2} \leq  2n\norm{y}^{2}+ 2n^{3}\norm{y}^{2}.$$

\subsection{Proof of Lemma~\ref{lemma:beta_ineq}}

\noindent {\bf Lemma~\ref{lemma:beta_ineq}.\;}\textit{For all $\tau
  \in \mathbb{R}^{p}_{+}$ and $\lambda \in \mathbb{R}^{m+1}_{+}$,}
\begin{itemize}
\item[(1)] $\tr(\Var[\beta|\tau,\lambda]) \leq \lambda_{0}^{-1}\sum_{j=1}^{p}\tau_{j} +c^{*}\,\tr(ZZ^{T})\sum_{i=1}^{m}\lambda_{i}^{-1}$ , \hspace*{1mm} \mbox{and}

\item[(2)] $\norm{\Ex[\beta|\tau,\lambda]}^{2} \leq
  c^{*}\norm{y}^{2}n^{2}\left(s_{\max}^{2}\sum_{j=1}^{p}\tau_{j}+1\right)$ ,
\end{itemize}
\textit{where $s_{\max}$ is the largest singular value of $X$ and
  $c^{*}$ is a finite positive constant.}

\begin{proof}[Proof of Lemma~\ref{lemma:beta_ineq}]
From (\ref{eq:theta_cond_var})
\begin{equation}
\label{eq:beta_lem_eq1}
\tr(\Var[\beta|\tau,\lambda]) =  \tr(\tsli) +\tr(\lambda_{0}^{2}\tsli X^{T}Z\qsli Z^{T}X\tsli) \;.
\end{equation}
Notice that that
\begin{equation*}
\tsli = [\lambda_{0}(X^{T}X +D_{\tau}^{-1})]^{-1} \preceq \lambda_{0}^{-1}\left(D_{\tau}^{-1}\right)^{-1} = \lambda_{0}^{-1}D_{\tau} \;,
\end{equation*}
and thus
\begin{equation}
\label{eq:beta_lem_ineq2}
\tr(\tsli) \leq \lambda_{0}^{-1}\,\tr\left(D_{\tau}\right) = \lambda_{0}^{-1}\sum_{j=1}^{p}\tau_{j} \;.
\end{equation}
Next,
\begin{equation}
\label{eq:beta_lem_ineq3}
\begin{aligned}
\tr(\lambda_{0}^{2}\tsli X^{T}Z\qsli Z^{T}X\tsli) &= \norm{\qsl^{-1/2}Z^{T}(\lambda_{0}X\tsli)}^{2}\\
&\leq \norm{\qsl^{-1/2}Z^{T}}^{2}\norm{\lambda_{0}X\tsli}^{2}\\
&= \tr(Z\qsli Z^{T})\tr(\lambda_{0}^{2}\tsli X^{T}X\tsli)\\
&= \tr\left(Z(\lambda_{0}Z^{T}\msl Z^{T}+\Lambda)^{-1}Z^{T}\right)\tr(\lambda_{0}^{2}X\tsl^{-2}X^{T})\\
&\leq \lambda_{\min}^{-1}\tr(ZZ^{T})\norm{\lambda_{0}\tsli X^{T}}^{2}\\
&\leq c^{*}\,\tr(ZZ^{T})\sum_{i=1}^{m}\lambda_{i}^{-1}
\;,
\end{aligned}
\end{equation}
where the last inequality follows from (1) of Lemma~\ref{lemma:exp_beta_lem} since there exists some finite $c^{*}$ such that $\norm{\lambda_{0}\tsli X^{T}}^{2}\leq c^{*}$.
Thus, from (\ref{eq:beta_lem_ineq2}) and (\ref{eq:beta_lem_ineq3})
\begin{equation*}
\tr(\Var[\theta|\tau,\lambda]) \leq \lambda_{0}^{-1}\sum_{j=1}^{p}\tau_{j} +c^{*}\,\tr(ZZ^{T})\sum_{i=1}^{m}\lambda_{i}^{-1} \;,
\end{equation*}
which proves (1).

To prove (2), it follows from (\ref{eq:theta_cond_mean})
\begin{equation}
\label{eq:beta_lem_eq4}
\begin{aligned}
\norm{\Ex[\beta|\tau,\lambda]}^{2} &= \norm{\lambda_{0}\tsli X^{T}y - \lambda_{0}^{2}\tsli X^{T}Z\qsli Z^{T}\msl y}^{2}\\
&= \norm{\lambda_{0}\tsli X^{T}(I - \lambda_{0}Z\qsli Z^{T}\msl)y}^{2}\\
&\leq \norm{\lambda_{0}\tsli X^{T}}^{2}\norm{I - \lambda_{0}Z\qsli Z^{T}\msl}^{2}\norm{y}^{2}
\;,
\end{aligned}
\end{equation}
and from (1) and (2) of Lemma~\ref{lemma:exp_beta_lem},
\begin{equation}
\label{eq:beta_lem_eq5}
\norm{\Ex[\beta|\tau,\lambda]}^{2} \leq c^{*}\norm{y}^{2}n^{2}\left(s_{\max}^{2}\sum_{j=1}^{p}\tau_{j}+1\right) \;.
\end{equation}
\end{proof}

\subsection{Proof of Lemma~\ref{lemma:frac_beta_lem}}

\noindent
{\bf Lemma~\ref{lemma:frac_beta_lem}.\;}\textit{For all $\tau\in\mathbb{R}_{+}^{p}$ and $\lambda\in\mathbb{R}_{+}^{m+1}$,}
$$\Ex\left[\sum_{j=1}^{p}\frac{1}{\abs{\beta_{j}}^{\nu(c)}}\Big|\tau,\lambda\right] \leq p\kappa(c)s_{\max}^{\nu(c)}\,\lambda_{0}^{\nu(c)/2} + \kappa(c)\lambda_{0}^{\nu(c)/2}\sum_{j=1}^{p}\frac{1}{\tau_{j}^{\nu(c)/2}} \;,$$
\textit{where}
$$\kappa(c):= \frac{\Gamma\left(\frac{1-\nu(c)}{2}\right)2^{\frac{1-\nu(c)}{2}}}{\sqrt{2\pi}} \;,$$
\textit{and $s_{\max}$ is the largest sinular value of the matrix $X$.}

\begin{proof}[Proof of Lemma~\ref{lemma:frac_beta_lem}]
Recall that given $\tau$ and $\lambda$, $\beta \sim N(\mu,\Sigma)$ where $\mu$ and $\Sigma$ are given by (\ref{eq:theta_cond_mean}) and (\ref{eq:theta_cond_var}), respectively. Thus, $\beta_{j}\sim N(\mu_{j},\sigma_{j}^{2})$ where $\mu_{j} = e_{j}^{T}\mu$ and $\sigma_{j}^{2} = e_{j}^{T}\Sigma e_{j}$ for $j=1,\ldots,p$. As in Lemma~\ref{lemma:exp_beta_lem}, $e_{1},\ldots,e_{p}$ represent the standard unit vectors for $\mathbb{R}^{p}$. From Proposition A1 of \cite{pal2014geometric}, it follows that
\begin{equation}
\label{eq:frac_beta_exp}
\Ex\left[\frac{1}{\abs{\beta_{j}}^{\nu(c)}}\Big|\tau,\lambda\right] \leq \frac{\kappa(c)}{\sigma_{j}^{\nu(c)}}, \ \text{for } j=1,\ldots,p \;,
\end{equation}
where
$$\kappa(c):= \frac{\Gamma\left(\frac{1-\nu(c)}{2}\right)2^{\frac{1-\nu(c)}{2}}}{\sqrt{2\pi}} \;.$$
From (\ref{eq:theta_cond_var}),
\begin{align*}
\frac{1}{\sigma_{j}^{\nu(c)}} &= \left(\frac{1}{e_{j}^{T}[\tsl^{-1} + \lambda_{0}^{2}\tsl^{-1}X^{T}Z\qsl^{-1}Z^{T}X\tsl^{-1}]e_{j}}\right)^{\nu(c)/2}\\
&\leq \left(\frac{1}{e_{j}^{T}\tsl^{-1}e_{j}}\right)^{\nu(c)/2}\\
&= \left(\frac{1}{e_{j}^{T}[\lambda_{0}(X^{T}X + D_{\tau}^{-1})]^{-1}e_{j}}\right)^{\nu(c)/2}\\
&= \left(\frac{1}{e_{j}^{T}[\lambda_{0}(V_{X}\Gamma_{X}^{T}U_{X}^{T}U_{X}\Gamma_{X}V^{T}_{X} + D_{\tau}^{-1})]^{-1}e_{j}}\right)^{\nu(c)/2}\\
&= \left(\frac{1}{e_{j}^{T}[\lambda_{0}(V_{X}\Gamma_{X}^{T}\Gamma_{X}V^{T}_{X} + D_{\tau}^{-1})]^{-1}e_{j}}\right)^{\nu(c)/2}\\
&\leq \left(\frac{1}{e_{j}^{T}[\lambda_{0}(s^{2}_{\max}I + D_{\tau}^{-1})]^{-1}e_{j}}\right)^{\nu(c)/2}\\
&= \lambda_{0}^{\nu(c)/2}\left(s_{\max}^{2} + \frac{1}{\tau_{j}}\right)^{\nu(c)/2}\\
&\leq (s_{\max}^{2})^{\nu(c)/2}\lambda_{0}^{\nu(c)/2} + \lambda_{0}^{\nu(c)/2}\frac{1}{\tau_{j}^{\nu(c)/2}}
\;,
\end{align*}
where $s_{\max}$ is the largest singular value of $X$, and the last inequality follows from the fact that $(x+y)^{\delta}\leq x^{\delta}+y^{\delta}$ for $\delta \in (0,1)$. Thus, from (\ref{eq:frac_beta_exp})
\begin{equation}
\label{eq:new_drift_eq3}
\Ex\left[\sum_{j=1}^{p}\frac{1}{\abs{\tilde{\beta}_{j}}^{c}}\bigg|\tau,\lambda\right] \leq\sum_{j=1}^{p}\frac{\kappa(c)}{\sigma_{j}^{\nu(c)}} \leq p\kappa(c)s_{\max}^{\nu(c)}\,\lambda_{0}^{\nu(c)/2} + \kappa(c)\lambda_{0}^{\nu(c)/2}\sum_{j=1}^{p}\frac{1}{\tau_{j}^{\nu(c)/2}} \;.
\end{equation}

\end{proof}

\subsection{Proof of Lemma~\ref{lemma:u_ineq}}

\noindent {\bf Lemma~\ref{lemma:u_ineq}.\;}\textit{Suppose that $Z$
  has full column rank. For all $\tau \in \mathbb{R}^{p}_{+}$ and
  $\lambda \in \mathbb{R}^{m+1}_{+}$,}
\begin{itemize}
\item[(1)] $\tr(R_{i}\qsli R_{i}^{T}) \leq q_{i}\lambda_{i}^{-1}$ ,
  \hspace*{1mm} \mbox{and}

\item[(2)] $||\Ex[u_{i}|\tau,\lambda]||^{2} \leq
  q_{i}\tr[(Z^{T}Z)^{-1}]n^{3}\norm{y}^{2}\left(s_{\max}^{2}\sum_{j=1}^{p}\tau_{j}+1\right)$
  ,
\end{itemize}
\textit{for $i=1,\ldots,m$ .}
\begin{proof}[Proof of Lemma~\ref{lemma:u_ineq}]
Note that
\begin{equation}
\label{eq:u_lem_eq0}
\qsli = (\lambda_{0}Z^{T}\msl Z +\Lambda)^{-1} \preceq \Lambda^{-1} \;,
\end{equation}
thus
\begin{equation}
\label{eq:u_lem_eq1}
\tr(R_{i}\qsli R_{i}^{T}) \leq \tr(R_{i}\Lambda^{-1}R_{i}^{T}) = q_{i}\lambda_{i}^{-1}, \ \  i=1,\ldots,m \;,
\end{equation}
which proves the first result.

Next, from (\ref{eq:theta_cond_mean}),
\begin{equation}
\label{eq:u_lem_eq2}
\begin{aligned}
\norm{\Ex[u_{i}|\tau,\lambda]}^{2} &= \norm{\Ex[R_{i}u|\tau,\lambda]}^{2}\\
&\leq \norm{R_{i}}^{2}\norm{\Ex[u|\tau,\lambda]}^{2}\\
&=\tr(I_{q_{i}})\norm{\lambda_{0}\qsli Z^{T}\msl y}^{2}\\
&=q_{i}\norm{\lambda_{0}\qsli Z^{T}\msl y}^{2}\\
&=q_{i}\norm{\lambda_{0}(Z^{T}Z)^{-1}Z^{T}Z\qsli Z^{T}\msl y}^{2}\\
&\leq q_{i}\norm{(Z^{T}Z)^{-1}Z^{T}}^{2}\norm{\lambda_{0}Z\qsli Z^{T}\msl y}^{2}\\
&= q_{i}\norm{(Z^{T}Z)^{-1}Z^{T}}^{2}\norm{\lambda_{0}\msl^{-1/2}\msl^{1/2}Z\qsli Z^{T}\msl y}^{2}\\
&\leq q_{i}\norm{(Z^{T}Z)^{-1}Z^{T}}^{2}\norm{\msl^{-1/2}}^{2}\norm{\lambda_{0}\msl^{1/2}Z\qsli Z^{T}\msl^{1/2}}^{2}\norm{\msl^{1/2}}^{2}\norm{y}^{2}\\
&= q_{i}\tr((Z^{T}Z)^{-1}Z^{T}Z(Z^{T}Z)^{-1})\tr(\msl^{-1})\norm{\lambda_{0}\msl^{1/2}Z\qsli Z^{T}\msl^{1/2}}^{2}\tr(\msl)\norm{y}^{2}\\
&= q_{i}\tr((Z^{T}Z)^{-1})\tr(\msl^{-1})\norm{\lambda_{0}\msl^{1/2}Z\qsli Z^{T}\msl^{1/2}}^{2}\tr(\msl)\norm{y}^{2} \;.
\end{aligned}
\end{equation}
Thus, from (2) of  Lemma~\ref{lemma:m_mat_lem} and (2) of Lemma~\ref{lemma:ZQ_lemma},
\begin{equation}
\label{eq:u_lem_eq3}
\begin{aligned}
\norm{\Ex[u_{i}|\tau,\lambda]}^{2} &\leq q_{i}\tr[(Z^{T}Z)^{-1}](\tau_{\max}s_{\max}^{2}+1)\tr(I_{n})n\tr(I_{n})\norm{y}^{2}\\
&= q_{i}\tr[(Z^{T}Z)^{-1}]n^{3}\norm{y}^{2}(\tau_{\max}s_{\max}^{2}+1)\\
&=q_{i}\tr[(Z^{T}Z)^{-1}]n^{3}\norm{y}^{2}\left(s_{\max}^{2}\sum_{j=1}^{p}\tau_{j}+1\right),
\end{aligned}
\end{equation}
which proves the second result.
\end{proof}

\subsection{Proof of Lemma~\ref{lemma:tau_exp_lem}} {\bf
  Lemma~\ref{lemma:tau_exp_lem}.\;}\textit{For all
  $(\theta,\lambda)\in\mathsf{X}$ ,}
\begin{itemize}
\item[1.] $\Ex[\tau_{j}|\theta,\lambda] \leq \frac{4c+1}{4d} +
  \frac{\lambda_{0}\beta_{j}^{2}}{2} \;, \hspace*{1mm} \mbox{and}$
    \item[2.] $\Ex[\tau_{j}|\theta,\lambda] \leq \frac{c}{d}+\frac{\beta_{j}^{2}}{2C}+\frac{\lambda_{0}C}{4d}$ .
\end{itemize}
\textit{for every $C > 0$ .}
\begin{proof}[Proof of Lemma~\ref{lemma:tau_exp_lem}]
From (\ref{eq:tau_exp_ineq2}),
\begin{equation}
\Ex[\tau_{j}|\theta,\lambda] = \sqrt{\frac{\lambda_{0}\beta_{j}^{2}}{2d}}\,\frac{K_{c+\frac{1}{2}}\left(\sqrt{2d\lambda_{0}\beta_{j}^{2}}\right)}{K_{c-\frac{1}{2}}\left(\sqrt{2d\lambda_{0}\beta_{j}^{2}}\right)} \;.
\end{equation}
Theorem 2 of \cite{segura2011bounds} states that
\begin{equation*}
\frac{K_{\nu-\frac{1}{2}}(x)}{K_{\nu+\frac{1}{2}}(x)} > \frac{x}{\sqrt{x^{2}+\nu^{2}}+\nu}\;,
\end{equation*}
for $\nu>0$ and $x>0$.

Thus,
\begin{equation}
\begin{aligned}
\Ex[\tau_{j}|\theta,\lambda] &< \sqrt{\frac{\lambda_{0}\beta_{j}^{2}}{2d}}\times\frac{\sqrt{c^{2} + 2d\lambda_{0}\beta_{j}^{2}}+c}{\sqrt{2d\lambda_{0}\beta_{j}^{2}}} \\
&= \frac{\sqrt{c^{2} + 2d\lambda_{0}\beta_{j}^{2}}+c}{2d} \\
&\leq \frac{c}{2d} + \frac{c}{2d} + \sqrt{\frac{\lambda_{0}\beta_{j}^{2}}{2d}}\\
&= \frac{c}{d} + \sqrt{\frac{\lambda_{0}\beta_{j}^{2}}{2d}} \;,
\end{aligned}
\end{equation}
where we make use of the fact that $\sqrt{x^{2}+y^{2}} \leq \abs{x} + \abs{y}$.

To get the first inequality, we use the fact that $\abs{xy}\leq(x^{2}+y^{2})/2$. Thus
\begin{equation}
\begin{aligned}
\Ex[\tau_{j}|\theta,\lambda] &\leq \frac{c}{d} + \frac{\lambda_{0}\beta_{j}^{2}}{2} + \frac{1}{4d} \\
&= \frac{4c + 1}{4d} + \frac{\lambda_{0}\beta_{j}^{2}}{2} \;.
\end{aligned}
\end{equation}
Similarly, for any constant $C > 0$,
\begin{equation}
\begin{aligned}
\Ex[\tau_{j}|\theta,\lambda] &\leq \frac{c}{d} + \sqrt{\frac{\lambda_{0}\beta_{j}^{2}}{2d}} \\
&= \frac{c}{d} + \sqrt{\frac{\lambda_{0}\beta_{j}^{2}C}{2dC}} \\
&\leq \frac{c}{d} + \frac{\beta_{j}^{2}}{2C} + \frac{\lambda_{0}C}{4d} \;.
\end{aligned}
\end{equation}
\end{proof}

\subsection{Proof of Lemma~\ref{lemma:tau_exp_lem2}}

{\bf Lemma~\ref{lemma:tau_exp_lem2}.\;}\textit{For all
  $(\theta,\lambda)\in\mathsf{X}$,}
$$\Ex[\tau_{j}^{-1}|\theta,\lambda]  \leq d +\frac{3}{2\lambda_{0}\beta_{j}^{2}} \;.$$

\begin{proof}[Proof of Lemma~\ref{lemma:tau_exp_lem2}]
From (\ref{eq:tau_exp_ineq3}),
\begin{equation*}
\Ex[\tau_{j}^{-1}|\theta,\lambda] = \sqrt{\frac{2d}{\lambda_{0}\beta_{j}^{2}}}\,\frac{K_{c-\frac{3}{2}}\left(\sqrt{2d\lambda_{0}\beta_{j}^{2}}\right)}{K_{c-\frac{1}{2}}\left(\sqrt{2d\lambda_{0}\beta_{j}^{2}}\right)} \;.
\end{equation*}
From Lemma 2.2 of \cite{ismail1978monotonicity}, for each $x>0$, $s>0$ and $s_{1}\in\mathbb{R}$, the function $K_{s_{1}+s}(x)/K_{s_{1}}(x)$ is increasing in $s_{1}$. Thus, for $s_{1} < s_{2}$, $s >0$ and $x > 0$
\begin{equation*}
\frac{K_{s_{1}+s}(x)}{K_{s_{1}}(x)} \leq \frac{K_{s_{2}+s}(x)}{K_{s_{2}}(x)} \;.
\end{equation*}
Thus, taking $s_{1} = -\frac{3}{2}$, $s_{2}=-\frac{1}{2}$ and $s=c$, we have
\begin{equation*}
\frac{K_{c-\frac{3}{2}}(x)}{K_{-\frac{3}{2}}(x)} \leq \frac{K_{c-\frac{1}{2}}(x)}{K_{-\frac{1}{2}}(x)} \;.
\end{equation*}
Since $K_{s}(x) > 0$ for $s \in \mathbb{R}$ and $x > 0$ (\cite{abramowitz1966handbook}, page 374), it follows that
\begin{equation*}
\frac{K_{c-\frac{3}{2}}(x)}{K_{c-\frac{1}{2}}(x)} \leq \frac{K_{-\frac{3}{2}}(x)}{K_{-\frac{1}{2}}(x)} \;.
\end{equation*}
Next, using the fact that
\begin{equation*}
K_{-\frac{1}{2}}(x) = e^{-x}\sqrt{\frac{\pi}{2x}} \;,
\end{equation*}
and
\begin{equation*}
K_{-\frac{3}{2}}(x) = e^{-x}\sqrt{\frac{\pi}{2x}}\left(1+\frac{1}{x}\right),
\end{equation*}
 hence
$$\frac{K_{-\frac{3}{2}}(x)}{K_{-\frac{1}{2}}(x)} = \left(1+\frac{1}{x}\right) \;,$$
for all $x \in \mathbb{R}$.

Thus,
\begin{align*}
\Ex[\tau_{j}^{-1}|\theta,\lambda] &\leq  \sqrt{\frac{2d}{\lambda_{0}\beta_{j}^{2}}}\,\left(1+\frac{1}{\sqrt{2d\lambda_{0}\beta_{j}^{2}}}\right)\\
&= \sqrt{\frac{2d}{\lambda_{0}\beta_{j}^{2}}} + \frac{1}{\lambda_{0}\beta_{j}^{2}}\\
&\leq d + \frac{1}{2\lambda_{0}\beta_{j}^{2}} + \frac{1}{\lambda_{0}\beta_{j}^{2}}\\
&= d +\frac{3}{2\lambda_{0}\beta_{j}^{2}}\;,
\end{align*}
where the second inequality follows from the fact that $\abs{xy} \leq (x^{2}+y^{2})/2$.
\end{proof}

\section{Proof of Lemma~\ref{lemma:frac_tau_lem}}

{\bf Lemma~\ref{lemma:frac_tau_lem}.\;} \textit{For all
  $(\theta,\lambda)\in\mathsf{X} \;,$}
$$\Ex\left[\frac{1}{\tau_{j}^{\nu(c)/2}}\Big| \theta,\lambda\right] \leq M_{1}\frac{1}{\lambda_{0}^{\nu(c)/2}\abs{\beta_{j}}^{\nu(c)}} + M_{2} \;,$$
\textit{where $M_{1}$ is a positive constant such that $M_{1}\kappa(c)
  < 1$, and $M_{2}$ is a positive finite constant.}

\begin{proof}[Proof of Lemma~\ref{lemma:frac_tau_lem}]

  This follows directly from the arguments on pp. 613-616 and p. 618
  of \cite{pal2014geometric}.
\end{proof}

\end{appendix}

\pagebreak

\bibliographystyle{ims}
\bibliography{references}

\end{document}